\documentclass[12pt]{amsart}
\usepackage{txfonts}
\usepackage{amssymb}
\usepackage{eucal}
\usepackage{graphicx}
\usepackage{amsmath}
\usepackage{mathrsfs}
\usepackage{amscd}
\usepackage[all]{xy}           
\usepackage{amsfonts,latexsym}
\usepackage{xspace}
\usepackage{epsfig}
\usepackage{float}
\usepackage{color}
\usepackage{fancybox}
\usepackage{colordvi}
\usepackage{multicol}
\usepackage{colordvi}
\usepackage{tikz}
\usepackage{soul}
\usepackage{wasysym}
\usepackage{ulem}
\usepackage[active]{srcltx} 
\ifpdf
\usepackage[colorlinks,final,backref=page,hyperindex]{hyperref}
\else
\usepackage[colorlinks,final,backref=page,hyperindex,hypertex]{hyperref}
\fi
\usepackage{enumerate}
\usepackage{makecell}
\usepackage{comment}

\makeatletter

\newtheorem{theorem}{Theorem}[section]
\newtheorem{prop}[theorem]{Proposition}
\newtheorem{lemma}[theorem]{Lemma}

\newtheorem{prop-def}{Proposition-Definition}[section]

\newtheorem{property}{Property}[section]
\theoremstyle{definition}
\newtheorem{defn}[theorem]{Definition}
\newtheorem{remark}[theorem]{Remark}
\newtheorem{exam}[theorem]{Example}

\newcommand{\nc}{\newcommand}

\nc{\delete}[1]{{}}

\nc{\vsa}{\vspace{-.1cm}} \nc{\vsb}{\vspace{-.2cm}}
\nc{\vsc}{\vspace{-.3cm}} \nc{\vsd}{\vspace{-.4cm}}
\nc{\vse}{\vspace{-.5cm}}

	\nc{\mlabel}[1]{\label{#1}}  
	\nc{\mcite}[1]{\cite{#1}}  
	\nc{\mref}[1]{\ref{#1}}  
	\nc{\meqref}[1]{\eqref{#1}}
	\nc{\mbibitem}[1]{\bibitem{#1}} 

\delete{
	\nc{\mlabel}[1]{\label{#1}  
		{\hfill \hspace{1cm}{\small\tt{{\ }\hfill(#1)}}}}
	\nc{\mcite}[1]{\cite{#1}{\small{\tt{{\ }(#1)}}}}  
	\nc{\mref}[1]{\ref{#1}{{\tt{{\ }(#1)}}}}  
	\nc{\meqref}[1]{\eqref{#1}{{\tt{{\ }(#1)}}}}
	\nc{\mbibitem}[1]{\bibitem[\bf #1]{#1}} 
}

\nc{\diff}{d_{\mathfrak{X}_{\infty}}}
\nc{\qd}{d_{\frakX}}
\nc{\qp}{P_{\frakX}}
\nc{\difforg}{U_D}
\nc{\rbforg}{U_{RB}}
\nc{\multforg}{U_{M}}
\nc{\diffree}{{}_DF}
\nc{\difcof}{F_D}
\nc{\rbfree}{{}_{RB}F}
\nc{\rbcof}{F_{RB}}
\nc{\multfree}{{}_{M}F}
\nc{\multcof}{F_M}

\nc{\vep}{\varepsilon}
\nc{\bin}[2]{ (_{\stackrel{\scs{#1}}{\scs{#2}}})}  
\nc{\binc}[2]{(\!\! \begin{array}{c} \scs{#1}\\
		\scs{#2} \end{array}\!\!)}  
\nc{\bincc}[2]{  ( {\scs{#1} \atop
		\vspace{-1cm}\scs{#2}} )}  
\nc{\bs}{\bar{S}}
\nc{\la}{\longrightarrow}
\nc{\ot}{\otimes}
\nc{\rar}{\rightarrow}
\nc{\dar}{\downarrow}
\nc{\dap}[1]{\downarrow \rlap{$\scriptstyle{#1}$}}
\nc{\defeq}{\stackrel{\rm def}{=}}
\nc{\dis}[1]{\displaystyle{#1}}
\nc{\dotcup}{\ \displaystyle{\bigcup^\bullet}\ }
\nc{\hcm}{\ \hat{,}\ }
\nc{\hts}{\hat{\otimes}}
\nc{\hcirc}{\hat{\circ}}
\nc{\lleft}{[}
\nc{\lright}{]}
\nc{\curlyl}{\left \{ \begin{array}{c} {} \\ {} \end{array}
	\right .  \!\!\!\!\!\!\!}
\nc{\curlyr}{ \!\!\!\!\!\!\!
	\left . \begin{array}{c} {} \\ {} \end{array}
	\right \} }
\nc{\longmid}{\left | \begin{array}{c} {} \\ {} \end{array}
	\right . \!\!\!\!\!\!\!}
\nc{\ora}[1]{\stackrel{#1}{\rar}}
\nc{\ola}[1]{\stackrel{#1}{\la}}
\nc{\scs}[1]{\scriptstyle{#1}} \nc{\mrm}[1]{{\rm #1}}
\nc{\dirlim}{\displaystyle{\lim_{\longrightarrow}}\,}
\nc{\invlim}{\displaystyle{\lim_{\longleftarrow}}\,}
\nc{\dislim}[1]{\displaystyle{\lim_{#1}}} \nc{\colim}{\mrm{colim}}
\nc{\mvp}{\vspace{0.3cm}} \nc{\tk}{^{(k)}} \nc{\tp}{^\prime}
\nc{\ttp}{^{\prime\prime}} \nc{\svp}{\vspace{2cm}}
\nc{\vp}{\vspace{8cm}}
\nc{\modg}[1]{\!<\!\!{#1}\!\!>}
\nc{\intg}[1]{F_C(#1)}
\nc{\lmodg}{\!<\!\!}
\nc{\rmodg}{\!\!>\!}
\nc{\cpi}{\widehat{\Pi}}
\nc{\sha}{{\mbox{\cyr X}}}  
\nc{\ssha}{{\mbox{\cyrs X}}} 
\nc{\tsha}{{\mbox{\cyrt X}}}
\nc{\shpr}{\diamond}    
\nc{\labs}{\mid\!}
\nc{\rabs}{\!\mid}

\font\cyr=wncyr10
\font\cyrs=wncyr7
\font\cyrt=wncyr5

\nc{\ann}{\mrm{ann}}
\nc{\Aut}{\mrm{Aut}}
\nc{\can}{\mrm{can}}
\nc{\Cont}{\mrm{Cont}}
\nc{\rchar}{\mrm{char}}
\nc{\cok}{\mrm{coker}}
\nc{\dtf}{{R-{\rm tf}}}
\nc{\dtor}{{R-{\rm tor}}}

\nc{\Div}{{\mrm Div}}
\nc{\End}{\mrm{End}}
\nc{\Ext}{\mrm{Ext}}
\nc{\Fil}{\mrm{Fil}}
\nc{\Fr}{\mrm{Fr}}
\nc{\Frob}{\mrm{Frob}}
\nc{\Gal}{\mrm{Gal}}
\nc{\GL}{\mrm{GL}}
\nc{\Hom}{\mrm{Hom}}
\nc{\hsr}{\mrm{H}}
\nc{\hpol}{\mrm{HP}}
\nc{\id}{\mrm{id}}
\nc{\im}{\mrm{im}}
\nc{\incl}{\mrm{incl}}
\nc{\length}{\mrm{length}}
\nc{\mchar}{\rm char}
\nc{\mpart}{\mrm{part}}
\nc{\ql}{{\QQ_\ell}}
\nc{\rank}{\mrm{rank}}
\nc{\rcot}{\mrm{cot}}
\nc{\rdef}{\mrm{def}}
\nc{\rdiv}{{\rm div}}
\nc{\rtf}{{\rm tf}}
\nc{\rtor}{{\rm tor}}
\nc{\res}{\mrm{res}}
\nc{\SL}{\mrm{SL}}
\nc{\Spec}{\mrm{Spec}}
\nc{\tor}{\mrm{tor}}
\nc{\Tr}{\mrm{Tr}}
\nc{\tr}{\mrm{tr}}

\nc{\ab}{\mathbf{Ab}}
\nc{\Set}{\mathbf{Set}}
\nc{\DRB}{\mathbf{DRB}}
\nc{\OPA}{\mathbf{OPA}}
\nc{\Alg}{{\mathbf{Alg}}}
\nc{\ALG}{{\mathbf{ALG}}}
\nc{\RB}{\mathbf{RB}}
\nc{\RBA}{\mathbf{RBA}}
\nc{\PDRB}{\mathbf{PDRB}}
\nc{\bfk}{{\bf k}}
\nc{\bfone}{{\bf 1}}
\nc{\bfzero}{{\bf 0}}
\nc{\detail}{\marginpar{\bf More detail}
	\noindent{\bf Need more detail!}
	\svp}
\nc{\Diff}{\mathbf{Diff}}
\nc{\gap}{\marginpar{\bf Incomplete}\noindent{\bf Incomplete!!}
	\svp}
\nc{\FMod}{\mathbf{FMod}}
\nc{\Int}{\mathbf{Int}}
\nc{\Mon}{\mathbf{Mon}}
\nc{\Mult}{{\mathbf{Mlt}}}
\nc{\remarks}{\noindent{\bf Remarks: }}
\nc{\Rep}{\mathbf{Rep}}
\nc{\Rings}{\mathbf{Rings}}
\nc{\Sets}{\mathbf{Sets}}

\nc{\BA}{{\mathbb A}}
\nc{\CC}{{\mathbb C}}
\nc{\DD}{{\mathbb D}}
\nc{\EE}{{\mathbb E}}
\nc{\FF}{{\mathbb F}}
\nc{\GG}{{\mathbb G}}
\nc{\HH}{{\mathbb H}}
\nc{\LL}{{\mathbb L}}
\nc{\NN}{{\mathbb N}}
\nc{\PP}{{\mathbb P}}
\nc{\QQ}{{\mathbb Q}}
\nc{\RR}{{\mathbb R}}
\nc{\TT}{{\mathbb T}}
\nc{\VV}{{\mathbb V}}
\nc{\ZZ}{{\mathbb Z}}
\nc{\TP}{\widetilde{P}}
\nc{\lp}{\widehat{P}}
\nc{\lpt}{\widehat{P}}
\nc{\lqb}{\widehat{Q}}
\nc{\lqt}{\widehat{Q}^{\, \omega}}
\nc{\ld}{\hat{d}}
\nc{\ldt}{\hat{d}^{\, \omega}}
\nc{\lqst}{\hat{q}^{\, \omega}}
\nc{\lqs}{\hat{q}}
\nc{\lpp}{\overline{P'}}
\nc{\dee}{\mathrm{Deg}}
\nc{\Mpf}{\mathbf{M}}
\nc{\m}{\iota}

\nc{\ots}{\otimes}
\nc{\faa}{\mathfrak{a}_1}
\nc{\fbb}{\mathfrak{b}_1}
\nc{\fa}{{\mathfrak a}}
\nc{\fb}{\mathfrak{b}}
\nc{\frakB}{{\frak B}} \nc{\frakm}{{\frak
		m}} \nc{\frakM}{{\frak M}}
\nc{\frakp}{{\frak p}}
\nc{\frakS}{{\frak S}}
\nc{\frakA}{{\frak A}} \nc{\frakx}{{\frakx}}
\nc{\frakSO}{{\mathfrak S_\Omega}}
\nc{\frakMO}{\mathfrak M_\Omega}

\nc{\Dif}{\mathbf{DIF}}
\nc{\DIF}{\mathbf{DIF}}
\nc{\ADR}{\mathbf{ADR}}
\nc{\OA}{\mathbf{OA}}
\nc{\ODA}{\mathbf{ODA}}
\nc{\ORB}{\mathbf{ORB}}
\nc{\DaRB}{\mathbf{DaRB}}
\nc{\G}{\mathbf{G}}
\nc{\C}{\mathbf{C}}
\nc{\A}{\mathbf{A}}
\nc{\B}{\mathbf{B}}
\nc{\T}{\mathbf{T}}
\nc{\bz}{b}

\nc{\tred}[1]{\textcolor{red}{#1}} \nc{\tgreen}[1]{\textcolor{green}{#1}}
\nc{\tblue}[1]{\textcolor{blue}{#1}} \nc{\tpurple}[1]{\textcolor{purple}{#1}}

\nc{\li}[1]{\tred{#1}}
\nc{\lir}[1]{\tred{\underline{Li:} #1}}
\nc{\ani}[1]{\tblue{#1 }}
\nc{\anir}[1]{\tblue{\underline{Ani:}#1 }}
\nc{\slr}[1]{\tpurple{\underline{Shilong:}#1 }}

\nc{\pdrba}{para-differential Rota-Baxter algebra\xspace}
\nc{\pdrbas}{para-differential Rota-Baxter algebras\xspace}
\nc{\Pdrba}{Para-differential Rota-Baxter algebra\xspace}
\nc{\Pdrbas}{Para-differential Rota-Baxter algebras\xspace}

\nc{\addx}{d_c\xspace}
\nc{\da}{d_c\xspace}
\nc{\der}{\frac{{\rm d}}{{\rm d}x}}
\nc{\px}{\frac{{\rm \partial}}{{\rm \partial}x}}
\nc{\bxddx}{d_{bx}\xspace}
\nc{\dbx}{d_{bx}\xspace}
\nc{\addt}{d_c\xspace}
\nc{\bxddt}{d_{bx}\xspace}
\nc{\afx}{\left(\sum\limits_{k=0}^{n}a_kx^k\right)\xspace}
\nc{\aft}{\left(\sum\limits_{k=0}^{n}a_kt^k\right)\xspace}
\nc{\bfx}{\left(\sum\limits_{k=0}^{n}b_kx^k\right)\xspace}
\nc{\bft}{\left(\sum\limits_{k=0}^{n}b_kt^k\right)\xspace}
\nc{\intx}{\int_{0}^{x}\xspace}
\nc{\free}[1]{\cald_\lambda\langle{#1}\rangle}
\nc{\freec}[1]{\cald_\lambda[{#1}]}
\nc{\freed}[1]{\cald_\lambda({#1})}
\nc{\freecd}[1]{\calc\cald_\lambda({#1})}
\nc{\freel}[1]{\call_\lambda({#1})}
\newcommand{\bk}{{\mathbf{k}}}

\nc{\Diffl}{\mathsf{DA}_\lambda}
\nc{\diffo}{{\mathsf{DO}_\lambda}}
\nc{\pd}{{\mathrm{PD}}}
\nc{\dRB}{{\mathrm{\Phi}_\mathsf{DRB}}}
\nc{\udRB}{{\mathrm{\Phi}_\mathsf{uDRB}}}
\nc{\OdRB}{{\mathrm{\Phi}_\mathsf{DRB}^0}}
 \nc{\OudRB}{{\mathrm{\Phi}_\mathsf{uDRB}^0}}
\nc{\inte}{{\mathrm{\Phi}_\mathsf{ID}}}
\nc{\uinte}{{\mathrm{\Phi}_\mathsf{uID}}}
\nc{\Ointe}{{\mathrm{\Phi}_\mathsf{ID}^0}}
 \nc{\Ouinte}{{\mathrm{\Phi}_\mathsf{uID}^0}}

\nc{\alg}{\mathsf{Alg}}
\delete{
\nc{\End}{\mrm{End}}
\nc{\Ext}{\mrm{Ext}}
\nc{\Fil}{\mrm{Fil}}
\nc{\Fr}{\mrm{Fr}}
\nc{\Frob}{\mrm{Frob}}
\nc{\Gal}{\mrm{Gal}}
\nc{\GL}{\mrm{GL}}
\nc{\Hom}{\mrm{Hom}}
}
\nc{\Hoch}{\mrm{Hoch}}
\nc{\Id}{\mrm{Id}}
\nc{\ID}{\mrm{ID}}
\nc{\Irr}{\mrm{Irr}}
\nc{\NLSW}{\mrm{NLSW}}
\nc{\Lie}{\mrm{Lie}}
\nc{\Nij}{\mrm{Nij}}
\nc{\Rey}{\mrm{Rey}}
\delete{
\nc{\rtf}{{\rm tf}}
\nc{\rtor}{{\rm tor}}
\nc{\res}{\mrm{res}}
\nc{\SL}{\mrm{SL}}
\nc{\Spec}{\mrm{Spec}}
\nc{\tor}{\mrm{tor}}
\nc{\Tr}{\mrm{Tr}}
\nc{\tr}{\mrm{tr}}
}
\nc{\wt}{\mrm{wt}}
\def\ot{\otimes}

\nc{\udl}{{\mathrm{udl}}}

\delete{
\nc{\bfk}{{\bf k}}
\nc{\bfone}{{\bf 1}}
\nc{\bfzero}{{\bf 0}}
\nc{\detail}{\marginpar{\bf More detail}
	\noindent{\bf Need more detail!}
	\svp}
\nc{\gap}{\marginpar{\bf Incomplete}\noindent{\bf Incomplete!!}
	\svp}
\nc{\FMod}{\mathbf{FMod}}
\nc{\Int}{\mathbf{Int}}
\nc{\remarks}{\noindent{\bf Remarks: }}
}
\nc{\zhx}{\text{-}}

\delete{
\nc{\BA}{{\mathbb A}}   \nc{\CC}{{\mathbb C}}
\nc{\DD}{{\mathbb D}}   \nc{\EE}{{\mathbb E}}
\nc{\FF}{{\mathbb F}}   \nc{\GG}{{\mathbb G}}
\nc{\HH}{{\mathbb H}}   \nc{\LL}{{\mathbb L}}
\nc{\NN}{{\mathbb N}}   \nc{\PP}{{\mathbb P}}
\nc{\QQ}{{\mathbb Q}}   \nc{\RR}{{\mathbb R}}
\nc{\TT}{{\mathbb T}}   \nc{\VV}{{\mathbb V}}
\nc{\ZZ}{{\mathbb Z}}   \nc{\TP}{\widetilde{P}}
}

\nc{\cala}{{\mathcal A}}    \nc{\calc}{{\mathcal C}}
\nc{\cald}{\mathcal{D}}     \nc{\cale}{{\mathcal E}}
\nc{\calf}{{\mathcal F}}    \nc{\calg}{{\mathcal G}}
\nc{\calh}{{\mathcal H}}    \nc{\cali}{{\mathcal I}}
\nc{\call}{{\mathcal L}}    \nc{\calm}{{\mathcal M}}
\nc{\caln}{{\mathcal N}}    \nc{\calo}{{\mathcal O}}
\nc{\calp}{{\mathcal P}}    \nc{\calr}{{\mathcal R}}
\nc{\cals}{{\mathcal S}}    \nc{\calt}{{\Omega}}
\nc{\calv}{{\mathcal V}}    \nc{\calw}{{\mathcal W}}
\nc{\calx}{{\mathcal X}}    \nc{\calu}{{\mathcal U}}
\nc{\caly}{{\mathcal Y}}

\nc{\msp}{ \ssha_\lambda  }
\nc{\uOpAlg}{{\mathfrak{uOpAlg}}}

\nc{\OpAlg}{{\mathfrak{OpAlg}}}

\nc{\ComOpAlg}{{\mathfrak{ComOpAlg}}}

\nc{\OpVect}{{\mathfrak{OpVect}}}

\nc{\OpSet}{{\mathfrak{OpSet}}}

\nc{\OpMon}{{\mathfrak{OpMon}}}

\nc{\ComOpMon}{{\mathfrak{ComOpMon}}}

\nc{\OpSem}{{\mathfrak{OpSem}}}

\nc{\ComOpSem}{{\mathfrak{ComOpSem}}}

\nc{\uAlg}{{\mathfrak{uAlg}}}

\nc{\ComAlg}{{\mathfrak{ComAlg}}}

\nc{\Vect}{{\mathfrak{Vect}}}

\nc{\ComMon}{{\mathfrak{ComMon}}}

\nc{\Sem}{{\mathfrak{Sem}}}

\nc{\mtOpSet}{{\Omega\zhx\mathfrak{Set}}}

\nc{\mtOpSem}{{\Omega\zhx\mathfrak{Sem}}}

\nc{\mtOpMon}{{\Omega\zhx\mathfrak{Mon}}}

\nc{\mtOpVect}{{\Omega\zhx\mathfrak{Vect}}}

\nc{\mtOpAlg}{{\Omega\zhx\mathfrak{Alg}}}

\nc{\mtuOpAlg}{{\Omega\zhx\mathfrak{uAlg}}}

\nc{\ComSem}{\mathfrak{ComSem}}

\nc{\rmI}{{\mathrm{I}}}
\nc{\rmII}{{\mathrm{II}}}
\nc{\rmIII}{{\mathrm{III}}}
\nc{\rmT}{{\mathrm{T}}}

\nc{\frakg}{{\frak g}}
\nc{\frakl}{{\frak l}}
\nc{\fraks}{{\frak s}}
\nc{\frakW}{{\frak W}}
\nc{\frakX}{{\frak X}}

\nc{\frakMstar}{{\mathfrak{M}_\Omega^\star}}
\nc{\frakSstar}{{\mathfrak{S}_\Omega^\star}}
\nc{\ilambda}{\lambda^{-1}}

\nc{\lc}{\lfloor}
\nc{\rc}{\rfloor}
\nc{\dl}{\mathrm{dlex}}
\nc{\glex}{\mathrm{glex}}
\nc{\diaglex}{\mathrm{diaglex}}
\nc{\deglex}{\mathrm{deglex}} \nc{\db}{\mathrm{db}} \nc{\lex}{\mathrm{lex}} \nc{\clex}{\mathrm{clex}} \nc{\dgp}{\mathrm{dgp}} \nc{\dgx}{\mathrm{dgx}} \nc{\br}{\mathrm{b}} \nc{\obd}{\mathrm{odb}} \nc{\ob}{\mathrm{ob}}
\nc{\dz}{\mathrm{dz}}
\nc{\ed}{\mathrm{ed}}
\nc{\p}{\mathrm{p}}
\nc{\dd}{\mathrm{d}}
\nc{\dlex}{\mathrm{Dlex}}
\nc{\zdp}{\mathrm{dp}}
\nc{\bre}[1]{|#1|}
\nc\kdot{\bfk\hspace{-2pt}\cdot\hspace{-2pt}}
\nc{\medmid}{{\,~{\tiny \longmid}~\,}}
\nc{\modk}{{\rm\ mod\,}}
\nc{\brp}{\mathrm{brp}}
\nc{\brd}{\mathrm{brd}}
\nc{\mpu}{u^{\ast}}
\nc{\mpv}{v^{\ast}}
\nc{\mpw}{w^{\ast}}
\nc{\mpx}{x^{\ast}}
\nc{\dps}{{\bf \dotplus}}
\nc{\dx}{d_\frakX}
\nc{\fx}{\mathfrak{X}_{\infty}}
\nc{\dep}{\mathrm{dep}}
\nc{\pfx}{P_\frakX}

\allowdisplaybreaks

\begin{document}
	\title[Para-differential Rota-Baxter algebras]{Para-differential Rota-Baxter algebras \\ and their free objects by Gr\"obner-Shirshov bases}
\author{Li Guo}
\address{Department of Mathematics and Computer Science,
	Rutgers University,
	Newark, NJ 07102, USA}
\email{liguo@rutgers.edu}

\author{Aniruddha Talele}
\address{Department of Mathematics and Computer Science,
	Rutgers University,
	Newark, NJ 07102, USA}
\email{aniruddha.talele8@gmail.com}

\author{Shilong Zhang}
\address{College of Science, Northwest A$\&$F University, Yangling 712100, Shaanxi, China}
\email{shlz@nwafu.edu.cn}

\author{Shanghua Zheng}
\address{School of Mathematics and Statistics, Jiangxi Normal University, Nanchang, Jiangxi 330022, China}
\email{zhengsh@jxnu.edu.cn}

\date{\today}
\begin{abstract}
The algebraic formulation of the derivation and integration related by the First Fundamental Theorem of Calculus (FFTC) gives rise to the notion of differential Rota-Baxter algebra. The notion has a remarkable list of categorical properties, in terms of the existence of (co)extensions of differential and Rota-Baxter operators, of the lifting of monads and comonads, and of mixed distributive laws. Conversely, using these properties as axioms leads to a class of algebraic structures called \pdrbas. 

This paper carries out a systematic study of \pdrbas. After their basic properties and examples from Hurwitz series and difference algebras, a Gr\"obner-Shirshov bases theory is established for \pdrbas. Then an explicit construction of free \pdrbas is obtained.
	\end{abstract}
	
\subjclass[2020]{
17B38, 
16Z10, 
18C15, 
	13N99, 
	16S10, 
	16W99 
	}
	
	\keywords{
	differential Rota-Baxter algebra, Rota-Baxter algebra, differential algebra, first foundamental theorem of calculus, Gr\"obner-Shirshov basis, free algebra, monad}
	
	\maketitle
	
	\tableofcontents
	
	\setcounter{section}{0}
	
	\allowdisplaybreaks
	
\section{Introduction}
This paper gives a systematic study of a class of algebraic structures originated from differential and integral analysis related by the First Fundamental Theorem of Calculus (FFTC), and then characterized by categorical properties. The free objects are constructed applying the method of Gr\"obner-Shirshov bases. 
\vspace{-.2cm}	
\subsection{Algebraic structures from the FFTC}
The algebraic study of analysis has a rich history. In his landmark work~\mcite{Ri} a century ago, Ritt  introduced  differential algebra to provide an algebraic framework for differential analysis and differential equations. A differential algebra is an associative algebra $R$ equipped with a linear operator $d$ that satisfies the following Leibniz rule.
\vspace{-.1cm}
	\begin{equation}
		d(xy)=d(x)y+xd(y)\ \text{ for all }\ x,\ y\in R.
		\mlabel{eq:der10}
	\end{equation}
Through the contributions of numerous mathematicians, including Kolchin, Kaplansky, and Singer, differential algebra has evolved into a substantial field encompassing differential Galois groups, differential algebraic groups, and differential algebraic geometry. This area has also found applications in number theory, logic, and the mechanical proof of mathematical theorems~\mcite{Kol,SP,Wu2}.
	
As an algebraic abstraction of the difference quotient
$(f(x+\lambda)-f(x))/\lambda$ for a nonzero scalar $\lambda$, a {\bf differential operator of weight $\lambda$}~\mcite{GK3} is defined to satisfy the equation
\vspace{-.1cm}
\begin{equation}
	d(xy)=d(x)y+xd(y)+\lambda d(x)d(y)\ \text{ for all }\ x,\ y\in R
	\mlabel{eq:der1}
\vspace{-.1cm}
\end{equation}
and $d(\bfone_{R})=0.$

On the other hand, the algebraic abstraction of integral analysis emerged as a byproduct of G. Baxter's probability study in 1960~\mcite{Ba}. A Baxter algebra, later known as a {\bf Rota-Baxter algebra}, is defined as an algebra $R$ equipped with a linear operator $P$ that satisfies the following identity:
\vspace{-.1cm}
	\begin{equation}
		P(x)P(y)=P(P(x)y)+P(xP(y))+\lambda P(xy)\ \text{ for all }\ x,\ y\in R.
		\mlabel{eq:bax1}
	\end{equation}
Here $\lambda$ is a fixed constant and is called the {\bf weight} of the Rota-Baxter operator.
When $\lambda=0$, the identity is just the integration-by-parts formula in a special case.

After the pioneering work of Cartier and Rota~\mcite{Ca,Ro} in combinatorics, the rapid recent developments of Rota-Baxter algebras have led to applications ranging from multiple zeta values in number theory to renormalization of perturbation quantum field theory~\mcite{Bai,CK,Gub,GK1,GZ,RR,Ro}.
	
The differential and integral operators are related by the
well-known First Fundamental Theorem of Calculus (FFTC).
As the algebraic abstraction of FFTC, the notion of a {\bf differential Rota-Baxter algebra with weight $\lambda$} was introduced~\mcite{GK3} as a triple $(R,d,P)$ such that
	\begin{enumerate}
\item
$(R,d)$ is	a differential algebra of weight $\lambda$,
		\item
$(R,P)$ is a Rota-Baxter algebra of weight $\lambda$,
\item	the following algebraic abstraction of the FFTC holds.
\vspace{-.2cm}
	\begin{equation}
		d P = \id_R.
		\mlabel{eq:fftc}
	\end{equation}
	\end{enumerate}
A variation of the differential Rota-Baxter algebra is the integro-differential algebra, studied in~\mcite{GGR,GRR,RR}.
\vspace{-.3cm}
\subsection{Categorical properties of the FFTC and related structures}

The abstraction of the FFTC as the differential Rota-Baxter algebra allows further algebraic study of the analyses. In this direction, a categorical approach was carried out in~\mcite{ZGK,ZGK2,ZGK3}, leading to the \pdrbas that will be studied in this paper. 

To give adequate justification for their study, we next provide a summary of the categorical investigation~\mcite{ZGK,ZGK2,ZGK3}, attracted interests from combinatorics, differential algebra, probability and computer science~\mcite{BCLS,CL,Fr,Ja,St}.
Readers who are not concerned about the motivation can safely skip to Section~\mref{ss:outline}. 

The categorical investigation~\mcite{ZGK,ZGK2,ZGK3} was carried out in three stages, one in each of the three articles.
\subsubsection{Stage 1: categorical implications of FFTC}
The finding in~\mcite{ZGK} is that this natural algebraic abstraction of the FFTC as differential Rota-Baxter algebra has distinguished categorical implications in terms of liftings of monads and mixed distributive laws.

To give a sketch of the main results in~\mcite{ZGK}, we fix an algebra $A$. Let $A^\NN$ denotes the Hurwitz series algebra over $A$~\mcite{Ke,GK3}. Equipped with a natural differential operator $\partial_A$, $A^\NN$ is the cofree differential algebra on $A$. Categorically, there is a comonad $\mathbf{C}$ giving differential algebras. On the other hand, let $\sha(A)$ be the mixable shuffle product algebra~\mcite{GK1}. Then with a naturally defined Rota-Baxter operator $P_A$, the pair $(\sha(A), P_A)$ is the free Rota-Baxter algebra on $A$. This leads to a monad $\mathbf{T}$ giving Rota-Baxter algebras. In~\mcite{GK3}, a differential operator on $A$ is uniquely extended to one on $\sha(A)$, enriching $\sha(A)$ to a differential Rota-Baxter algebra, giving the free differential Rota-Baxter algebra on $A$. This is interpreted as a lifting of the monad $\mathbf{T}$. Reversing the order, a Rota-Baxter operator on $A$ is uniquely extended to one on $A^\NN$, again enriching $A^\NN$ to a differential Rota-Baxter algebra, yielding the cofree differential Rota-Baxter algebra on $A$. This means that there is a lifting of the comonad $\mathbf{C}$. Further, by the lifting, we obtain a distributive law of the monad $\mathbf{T}$ over the comonad $\mathbf{C}$ and vice versa. 
	
In summary, the coupling of a differential operator and a Rota-Baxter operator via FFTC has the following properties. 
\begin{property}
{\bf (FFTC categorical properties)}:
\begin{enumerate}
\item extensions of differential operators,
\item coextensions (covers) of Rota-Baxter operators, 
\item liftings of the comonad $\mathbf{C}$, 
\item lifting of the monad $\mathbf{T}$, 
\item existence of a mixed distributive laws of $\mathbf{T}$ over $\mathbf{C}$. 
\end{enumerate}
\mlabel{p:cat}
\end{property}

\subsubsection{Stage 2: FFTC categorical properties for other constrains}

The above rich categorical properties of differential Rota-Baxter algebras defined by FFTC naturally leads one to ask whether there are analogs of FFTC that also satisfy the FFTC categorical properties in Property\,\mref{p:cat}. 

A follow-up study of \mcite{ZGK} was carried out in~\mcite{ZGK2} to address this question. There the operator identity in the FFTC is viewed as an instance of a polynomial identity in two noncommutative variables $x, y$ symbolizing the differential operator and Rota-Baxter operator respectively. Then a polynomial $\omega=\omega(x,y)\in \bfk\langle x,y\rangle$ is regarded as a constraint between a differential operators $x=d$ and a Rota-Baxter operator $y=P$ on an algebra $R$, with the instance $\omega(x,y)=xy-1$ giving the FFTC in Eq.~\meqref{eq:fftc}. 
In general, a triple $(R,d,P)$ is called a {\bf para-differential Rota-Baxter algebra of type $\omega$} if $x=d$ and $y=P$ satisfy the constrain given by the $\omega(x,y)$. 

The remarkable discovery of \mcite{ZGK2} is that, for a constrain $\omega$ in
\begin{equation}
\calt:=xy+\bfk[x]+y\bfk[x] = \left\{
		xy-(\phi(x) +y\psi(x))\,|\, \phi,\psi\in\bfk[x]\right\},
		\mlabel{eq:T}
	\end{equation}
the FFTC categorical properties in Property\,\mref{p:cat} are in fact {\it mutually equivalent}. In other words, para-differential Rota-Baxter algebras of a given type $\omega$ either satisfy all of the properties or none of the properties. 
See \cite[Theorem 3.15]{ZGK2} for the precise statement. 

\subsubsection{Stage 3: classification of FFTC type algebraic structures}

The next stage of the study, carried out in \mcite{ZGK3}, is to explicitly determine the types $\omega\in \calt$ that satisfy any, and hence all, of the FFTC categorical properties. 

Consider the following two subsets of $\calt$:
\begin{equation}
\calt_0:=\{xy-a_0\,|\, a_0 \in \bfk\}\cup \{xy-\left(by+yx\right)\,|\, b\in \bfk\}, \quad \calt_{\bfk}:=\{xy, xy-1, xy-yx\}.
\mlabel{eq:pdrbrel}
\end{equation}
It is shown in\,\mcite{ZGK3} that, for $\omega\in \Omega$, 
the type-$\omega$ para-differential Rota-Baxter algebra with weight zero (resp. arbitrary weight) has any, and hence all, of the FFTC categorical properties in Property\,\mref{p:cat} if and only if $\omega$ is in $\Omega_0$ (resp. in $\Omega_\bfk$).

\subsection{The study of FFTC type algebras and layout of the paper}
\mlabel{ss:outline}

Since the type-$\omega$ differential Rota-Baxter algebras classified in Eq.\,\meqref{eq:pdrbrel} are characterized by the same categorical properties as the differential Rota-Baxter algebra, they deserve a careful study as in the case of differential Rota-Baxter algebras initiated in~\mcite{GK3}. This is the purpose of the present paper.

Henceforth in this paper, by a {\bf \pdrba} we mean a \pdrba of type $\omega$ with $\omega$ in $\Omega_0$ or $\Omega_\bfk$.

Here is an outline of this paper.

In Section~\mref{s:propex}, we organize the \pdrbas into three types, in Definition~\mref{de:pdrbtype}. We then give some properties and examples of \pdrbas from Hurwitz series and difference algebras. 

Section~\mref{s:nogsb} introduces a new monomial order tailored to our study (Proposition\,\mref{prop:monom}), and gives a discussion of Gr\"obner-Shirshov bases for operated algebras (Theorem\,\mref{thm:CDL}).  

In Section~\mref{s:gsbqdrb}, the general theory in Section~\mref{s:nogsb} is applied to obtain Gr\"obner-Shirshov bases for \pdrbas (Theorem\,\mref{thm:phigsb}).

The Gr\"obner-Shirshov bases for \pdrbas in Section~\mref{s:gsbqdrb} provide a linear basis for the free (noncommutative) \pdrbas. In Section~\mref{s:explicit}, the multiplication, and differential and Rota-Baxter operators are concretely defined, thus providing an explicit construction of free \pdrbas (Theorems\,\mref{thm:main-type1} and \mref{thm:ncfqdba}).

\noindent
{\bf Notations. }	
Throughout the paper, we fix a commutative ring $\bfk$ with identity and an element $\lambda \in \bfk$. Unless otherwise noted, we work in the categories of associative $\bfk$-algebras with identity, with or without linear operators. All operators and tensor products are also taken over $\bfk$. Thus references to $\bfk$ will be suppressed unless doing so can cause confusion.
We let $\NN$ denote the additive monoid of natural numbers $\{0,1,2,\ldots\}$ and
	$\NN_+=\{ n\in \NN\mid n>0\}$ the positive integers.
	
\section{\Pdrbas}
\mlabel{s:propex}
We first give a classification of \pdrbas into three types that will be studied more carefully later in this paper. We then give some properties and examples of these \pdrbas from Hurwitz series and endo algebras. 

\subsection{Classification of \pdrbas}
As discussed before, a \pdrba is an algebra $R$ equipped with a differential operator $d$ and a Rota-Baxter operator $P$, both of weight $\lambda$, such that $d$ and $P$ satisfy one of the operator identities in the sets $\Omega_0$ if $\lambda=0$ and $\Omega_{\bfk}$ for arbitrary $\lambda\in \bfk$ in Eq.~\meqref{eq:pdrbrel}: 
\begin{equation}
	\calt_0:=\{dP-a_0\id\,|\, a_0 \in \bfk\}\cup \{dP-\left(bP+Pd\right)\,|\, b\in \bfk\}, \quad \calt_{\bfk}:=\{dP, dP-\id, dP-Pd\}.
	\mlabel{eq:pdrbrel2}
\end{equation}

In order to give a more focused and unified study of \pdrbas, we first provide a preliminary analysis and classification. 

First the case of $dP-\id$ in $\Omega_{\bfk}$ corresponds to the First Fundamental Theorem of Calculus and has been systematically treated in~\mcite{GK3}. Let us next consider $dP-a\id$ in $\Omega_0$ with $a\in \bfk$. In this case, the weight is zero.  When $a$ is nonzero, then this relation is the same as $a^{-1}dP-\id$. Since $a^{-1}d$ is still a differential operator of weight zero, we are reduced to the previous case of $dP=\id$ which has been treated. When $a=0$, then we have $dP=0$, which is very degenerated comparing to the other cases that we will consider.
Excluding these treated or trivial cases, we are left to consider the operator polynomial
\vspace{-.2cm}
\begin{equation}
	\mlabel{eq:maincase}
	dP-Pd-bP, \quad b\in \bfk,
\end{equation}
where when $b\neq 0$, we further require $\lambda=0$.

Thus our study will focus on \pdrbas of the following three types.

\begin{defn} \mlabel{de:pdrbtype}
Let $\omega=	dP-Pd-bP$ be as in Eq.~\meqref{eq:maincase}.	A type $\omega$ \pdrba is called
	\begin{enumerate}
		\item a {\bf \pdrba of type I} if the operators $d$ and $P$ have weight $\lambda=0$, and $b=0$;
		\item a {\bf \pdrba of type II} if $d$ and $P$ have weight $\lambda\neq 0$, and $b=0$;
		\item a {\bf \pdrba of type III} if $d$ and $P$ have weight $\lambda=0$, and $b\neq 0$.
	\end{enumerate}
\end{defn}

Clearly type I can be merged with type II by allowing arbitrary $\lambda\in \bfk$, and can be merged with type III by allowing arbitrary $\bz\in \bfk$. While merging type I and II is convenient for most of the study; for certain other aspects of the study, such as for Gr\"obner-Shirshov bases in Section~\mref{s:gsbqdrb}, it is necessary to separately consider the three disjoint cases.

\subsection{\Pdrbas from Hurwitz series}
For an algebra $A$, let $A^{\NN}$ denote the $\bfk$-module of functions $f:\NN \rar A$.
We also view $f \in A^{\NN}$ as a sequence $(f_n)=(f_0,f_1,\cdots)$ with $f_n \in A$ given
by $f_n := f(n)$ for all $n \in \NN$. Then the operations on $A^\NN$ are defined componentwise.
	Following~\mcite{GK3}, the {\bf $\lambda$-Hurwitz product} on $A^{\NN}$ is given by
	\begin{equation}
		(fg)_n = \sum_{k=0}^{n}\sum_{j=0}^{n-k} {n\choose k} {n-k\choose j}
		\lambda^{k}f_{n-j}g_{k+j},\quad  f, g\in R^\NN.
		\mlabel{eq:hurprod}
	\end{equation}
Define a linear map
$$\partial_A: A^{\NN}\to A^{\NN},\;\partial_A(f)_n= f_{n+1},\quad f\in A^{\NN}, n\in \NN.$$
\begin{prop}$($\cite[Proposition~2.7]{GK3}$)$
For an algebra $A$, the operated algebra $(A^{\NN},\partial_A)$
is a  differential algebra of weight $\lambda$.
\mlabel{prop:cofree}
\end{prop}	
The following result gives a class of \pdrbas. 

\begin{prop}
Let $(A,P)$ be a Rota-Baxter algebra of weight $\lambda$ and  let $(A^{\NN},\partial_A)$ be the  differential algebra of weight $\lambda$. 
\begin{enumerate}
    \item Let $\lambda\in \bfk$ and $b=0$. Define a linear map
    $$\lpt: A^{\NN}\to A^{\NN}, \quad \lpt(f):= (P(f_0),P(f_1),P(f_2),\cdots), \text{ that is, } 
    \lpt(f)_n:=P(f_n), \  n\in \NN.$$
    Then the triple $(A^{\NN},\partial_A,\lpt)$ is a  \pdrba of type I when $\lambda =0$ or type II when $\lambda\neq 0$.
    \mlabel{it:t1}
    \item Let $\lambda=0$ and $b\neq 0$. Define
    \begin{equation}
   \lpt: A^{\NN}\to A^{\NN}, \quad  (\lpt(f))_n:= \sum_{s=0}^n{n\choose s}b^{n-s}P(f_s). \label{eq:cofcommeq1}
    \end{equation}
Then the triple $(A^{\NN},\partial_A,\lpt)$ is a  \pdrba of type III. \mlabel{it:t2}
\end{enumerate}
  \mlabel{prop:maincommex}
\end{prop}

\begin{proof}
\meqref{it:t1} 
For all $f\in A^{\NN}$ and $n\in\NN$, we compute
$$(\partial_A\lpt)(f)=(P(f_1),P(f_2),\cdots)=(\lpt\partial_A)(f),$$
which establishes the commutation relation:
\begin{equation}
\partial_A\lpt=\lpt\partial_A.
\mlabel{eq:c1}
\end{equation}
To verify that $\lpt$ is a Rota-Baxter operator of weight $\lambda$, we prove by induction on $m\geq 0$ that for all $f,g\in A^\NN$: 
\begin{equation}
(\lpt(f)\lpt(g))_m=(\lpt(\lpt(f)g+f\lpt(g)+\lambda fg))_m.
\mlabel{eq:c2}
\end{equation}
For the base case of $m=0$, the identity holds since $P$ is a Rota-Baxter operator of weight $\lambda$ on $A$.

Assume that the identity holds for $m=k$. Then for $m=k+1$, we compute: 
\begin{align*}
&\left(\lpt(\lpt(f)g+f\lpt(g)+\lambda fg)\right)_{k+1}=\left(\partial_A(\lpt(\lpt(f)g+f\lpt(g)+\lambda fg))\right)_{k}\\
&=\left(\lpt(\partial_A(\lpt(f)g)+\partial_A(f\lpt(g))+\lambda \partial_A(fg))\right)_{k}\quad(\text{by~Eq}.~(\ref{eq:c1}))\\
&=\left(
\lpt((\partial_A\lpt)(f)g+\lpt(f)\partial_A(g)+\lambda (\partial_A\lpt)(f)\partial_A(g))
\right)_{k}\\
&\quad+\left(
\lpt(
\partial_A(f)\lpt(g)+f(\partial_A\lpt)(g)+\lambda \partial_A(f)\partial_A(\lpt(g))
)
\right)_{k}\\
&\quad+\left(
\lambda\lpt(
\partial_A(f)g+f \partial_A(g)+\lambda \partial_A(f)\partial_A(g)
)
\right)_{k}\quad(\text{by~Eq}.~(\ref{eq:der1}))\\
&=\left(
\lpt((\lpt\partial_A)(f)g+
\partial_A(f)\lpt(g)+
\lambda\partial_A(f)g )
\right)_{k}\\
&\quad+\left(
\lpt( \lpt(f)\partial_A(g) +f(\lpt\partial_A)(g)+ \lambda f \partial_A(g)
)
\right)_{k}\\
&\quad+\left(
\lambda\lpt( (\lpt\partial_A)(f)\partial_A(g) + \partial_A(f)(\lpt\partial_A)(g)
 +\lambda \partial_A(f)\partial_A(g)
)
\right)_{k}\quad(\text{by~Eq}.~(\ref{eq:c1}))\\
&=\left(
\lpt( \partial_A(f))\lpt(g)+ \lpt(f)\lpt(\partial_A(g)) +
\lambda\lpt( \partial_A(f))\lpt(\partial_A(g))
\right)_{k}\quad(\text{by the induction hypothesis} )\\
&=\left(
\partial_A(\lpt(f))\lpt(g)+ \lpt(f)\partial_A(\lpt(g)) +
\lambda\partial_A(\lpt(f))\partial_A(\lpt(g))
\right)_{k}\quad(\text{by~Eq}.~(\ref{eq:c1}))\\
&=\left(\partial_A( \lpt(f)
 \lpt(g))\right)_{k}\quad(\text{by~Eq}.~(\ref{eq:der1}))\\
&=\left( \lpt(f) \lpt(g)\right)_{k+1}.
\end{align*}
This completes the induction, proving that $\lpt$ is a Rota-Baxter operator of weight $\lambda$.

\meqref{it:t2} 
For any $f\in A^{\NN}$ and $n\in\NN$, we compute:
\begin{align*}
&\left((\partial_A\lpt)(f)\right)_n=\left( \lpt(f)\right)_{n+1}=\sum\limits_{s=0}^{n+1}{{n+1}\choose s}b^{n+1-s}P(f_s)\\
&= \sum\limits_{s=1}^{n+1} {{n}\choose {s-1}} b^{n+1-s}P(f_s)
+\sum\limits_{t=0}^{n} {{n}\choose {t}} b^{n+1-t}P(f_t)\\
&= \left((\lpt\partial_A)(f)+b\lpt(f)\right)_n.
\end{align*}
This establishes the relation:
\begin{equation}
\partial_A\lpt=\lpt\partial_A+b\lpt.
\mlabel{eq:b1}
\end{equation}
To show that $\lpt$ is a Rota-Baxter operator of weight $0$, we prove by induction on $m$ that for all $f,g\in A^\NN$:
\begin{equation}
\left(\lpt(f)\lpt(g)\right)_m=\left(\lpt(\lpt(f)g+f\lpt(g) )\right)_m.
\mlabel{eq:b2}
\end{equation}
The base case $m=0$ follows from the fact that $P$ is a Rota-Baxter operator of weight $0$.

Assume that the identity holds for $m=k$. Then for $m=k+1$, we have
{\small
\begin{align*}
&\left(\lpt(\lpt(f)g+f\lpt(g) ) \right)_{k+1}=\left(\partial_A(\lpt(\lpt(f)g+f\lpt(g) ))\right)_{k}\\
&= \left(\lpt(\partial_A(\lpt(f)g)+\partial_A(f\lpt(g)) )+b\lpt( \lpt(f)g + f\lpt(g)) \right)_{k}\quad(\text{by~Eq}.~(\ref{eq:b1}))\\
&= \left(
\lpt((\partial_A\lpt)(f)g+\lpt(f)\partial_A(g) )
\right)_{k} +\left(
\lpt(
\partial_A(f)\lpt(g)+f(\partial_A\lpt)(g)
)
\right)_{k} +\left(b\lpt( \lpt(f)g + f\lpt(g)) \right)_{k} \quad(\text{by~Eq}.~(\ref{eq:der10}))\\
&= \left(
\lpt((\lpt\partial_A)(f)g+
\partial_A(f)\lpt(g)  )
\right)_{k} +\left(
\lpt( \lpt(f)\partial_A(g) +f(\lpt\partial_A)(g) )
\right)_{k} +\left( 2b\lpt( \lpt(f)g + f\lpt(g))
\right)_{k}\quad(\text{by~Eq}.~(\ref{eq:b1}))\\
&= \left(
\lpt( \partial_A(f))\lpt(g)+ \lpt(f)\lpt(\partial_A(g)) +
2b\lpt(f)\lpt(g)
\right)_{k}\quad(\text{by the induction hypothesis} ) \\
&= \left(
\partial_A(\lpt(f))\lpt(g)+ \lpt(f)\partial_A(\lpt(g))
\right)_{k}\quad(\text{by~Eq}.~(\ref{eq:b1}))\\
&= \left(\partial_A( \lpt(f) \lpt(g))\right)_{k}\quad(\text{by~Eq}.~(\ref{eq:der10}))\\
&= \left( \lpt(f) \lpt(g)\right)_{k+1}.
\end{align*}
}
This completes the induction, proving that $\lpt$ is a Rota-Baxter operator of weight $0$. Thus $(A^{\NN},\partial_A,\lpt)$ is a \pdrba of type III.
\end{proof}

The close connection  between Hurwitz series and formal power series in~\cite{Ke} leads to the following concrete examples of \pdrbas of type I and III. For Examples of type II \pdrbas, see Section~\mref{ss:endo}.

\begin{exam}
(i) Let $C$ denote the $\RR$-algebra of continuous functions on $\RR$, and let $C[[t]]$ be the algebra of formal power series with coefficients in $C$. Equip $C[[t]]$ with the derivation 
$$\frac{d}{dt} \bigg(\sum\limits_{i=0}^\infty u_nt^n\bigg):=\bigg(\sum\limits_{i=0}^\infty nu_nt^{n-1}\bigg).$$
Then by \mcite{Ke}, there is an isomorphism of differential algebras of weight $0$:
\begin{equation} \mlabel{eq:diffiso}
T:(C^\NN, \partial_C)\rar \bigg(C[[t]],\frac{d}{dt}\bigg),\quad f\mapsto \sum\limits_{n=0}^\infty {f_{n}\over {n!}}t^n.
\end{equation}

The integration
$$\int:C\rar C, \quad \int(u): =\int_0^x u(s){\rm d}s$$
is a classical example of Rota-Baxter operator of weight $0$. For this operator, the Rota-Baxter operator $\widehat{\int}$ in Proposition~\mref{prop:maincommex}.\meqref{it:t1}, through the isomorphism $T$, gives rise to the Rota-Baxter operator
$$\overline{\int}:=T\circ \widehat{\int}\circ T^{-1}:C[[t]]\rar C[[t]], \quad \overline{\int}\bigg(\sum\limits_{n=0}^\infty u_nt^n\bigg):=\sum\limits_{n=0}^\infty \left(\int_0^x u_n(s){\rm d}s\right) t^n.$$
Then by Proposition~\ref{prop:maincommex}(\ref{it:t1}), the triple $\Big(C[[t]],\frac{d}{dt},\overline{\int}\Big)$ is a  \pdrba of type I.

\noindent
(ii)
Retain the same integration $\int$ and differential algebra $\Big(C[[t]],\frac{d}{dt}\Big)$ as above. The Rota-Baxter operator $\widehat{\int}$ on $C^{\NN}$ in Proposition\,\mref{prop:maincommex}.\meqref{it:t2}, through the isomorphism $T$, defines the Rota-Baxter operator
$$\overline{\int}:=T\circ \widehat{\int}\circ T^{-1}:C[[t]]\rar C[[t]], \quad \overline{\int}\bigg(\sum\limits_{n=0}^\infty u_nt^n\bigg):=\sum\limits_{n=0}^\infty \left(\sum\limits_{k=0}^n {n\choose k}b^{n-k}\int_0^x u_k(s){\rm d}s\right) t^n.
$$
By Proposition~\mref{prop:maincommex}(\ref{it:t2}), the triple $\Big(C[[t]],\frac{d}{dt},\overline{\int})\Big)$ is a  \pdrba of type III.
\end{exam}

Here are some more examples.

\begin{exam}
For  $0\neq b\in \bfk$, the linear operator 
$$\bxddx: \bfk[x] \rar \bfk[x], \quad \bxddx(f(x))=
bx\der (f(x)),$$
is a differential operator of weight $0$, and satisfies $ \bxddx \int-\int \bxddx =b\int$ with $\int(f(x))=\int_0^xf(t){\rm d}t$. Then  the triple $\left(\bfk[x], \dbx, \int\right)$ is a \pdrba of  type III.
\end{exam}

\begin{exam}
Let $\bfk[x,y]$ be the  polynomial algebra in two variables $x$ and $y$. By the Leibniz rule, the partial derivative $\px:\bfk[x,y]\rar\bfk[x,y]$ is a differential operator of weight $0$. Let $\int_y:\bfk[x,y]\rar\bfk[x,y]$ denote the integral with respect to $y$:
$$\int_y\bigg(\sum\limits_{m,n }a_{m,n}x^my^n\bigg)=\sum\limits_{m,n }a_{m,n}x^m\int_0^yt^n{\rm d}t=\sum\limits_{m,n }a_{m,n}x^m\frac{y^{n+1}}{n+1}\,\quad \text{for }\, \sum\limits_{m,n }a_{m,n}x^my^n\in \bfk[x,y].$$
By the integration by parts formula, $\int_y$ is a Rota-Baxter operator of weight $0$. Furthermore $\px\int_y=\int_y\px$ since the two operators act on different variables. Therefore, the triple $\Big(\bfk[x,y],\px,\int_y\Big)$ is a  \pdrba of type I.
\end{exam}

\subsection{\Pdrbas and endo (difference) algebras}
\mlabel{ss:endo}
The notion of difference algebras has been widely studied alongside with the differential algebras~\mcite{Cohn,Lev}. Recently, the more general notion of an endo algebras has appeared in the Lie algebra context from categorification of Lie bialgebras and the classical Yang-Baxter equation~\mcite{BGS}.

\begin{defn}
\mlabel{de:endo}
\begin{enumerate}
\item An {\bf endo algebra} is an algebra $R$ equipped with an endomorphism $\sigma$ of $R$. 
\item A {\bf difference algebra} is an endo algebra where the endomorphism is injective\,\mcite{Cohn,Lev}.
\item An {\bf endo Rota-Baxter algebra} of weight $1$ is a Rota-Baxter algebra $(R,P)$ of weight $1$ equipped with an endomorphism $\sigma$ of the Rota-Baxter algebra $(R,P)$, in the sense that $\sigma P=P\sigma$.
\item A {\bf Rota-Baxter operator} of weight $1$ on an endo algebra $(R,\sigma)$ is a Rota-Baxter operator on $R$ such that $P\sigma =\sigma P$.
\end{enumerate}
\end{defn}
Thus an endo Rota-Baxter algebra is the same as an endo algebra equipped with a Rota-Baxter operator.

It is direct to check that $d:R\to R$ is a differential operator of weight $1$ if and only if $\id+d$ is an endomorphism~\mcite{GK3}. Adding a Rota-Baxter operator gives the following direct consequence.

\begin{prop} \mlabel{p:endopdrb}
Let $R$ be an algebra and $P, d:R\to R$ be linear maps.
The following statements are equivalent.
\begin{enumerate}
\item The triple $(R,P,d)$ is a \pdrba of type II with weight $1$;	
\item The pair $(R,P)$ is a Rota-Baxter algebra of weight $1$ and $\id+d$ is an endomorphism of $(R,P)$;
\item The pair $(R,\id+d)$ is an endo algebra and $P$ is a Rota-Baxter operator on $(R,\id+d)$.
\end{enumerate}
\end{prop}

This equivalence gives another source to obtain \pdrbas of type II.
Here is a specific example.
\begin{exam}Let $\bfk$ be an algebraically closed field of characteristic zero. Let $M^u_2(\bfk)$ denote the $2\times 2$ upper triangular  matrix algebra. Define
	\vspace{-.2cm}
\begin{equation*}
P:M^u_2(\bfk)\to M^u_2(\bfk),\qquad \left(\begin{array}{cc}
a_{11} &a_{12}\\
0& a_{22}
\end{array}\right)\mapsto  \left(\begin{array}{cc}
a_{22}-a_{11} &-a_{12}\\
0& 0
\end{array}\right).
\end{equation*}
Then by \cite[Corollary~3.6]{GG20}, $P$ is a Rota-Baxter operator of weight $1$.
On the other hand, we set
\vspace{-.2cm}
\begin{equation*}
\sigma:M^u_2(\bfk)\to M^u_2(\bfk),\qquad \left(\begin{array}{cc}
a_{11} &a_{12}\\
0& a_{22}
\end{array}\right)\mapsto  \left(\begin{array}{cc}
a_{11} &a_{12}+a_{22}-a_{11}\\
0& a_{22}
\end{array}\right).
\end{equation*}
Then by a direct computation, $\sigma$ is an injective endomorphism.
Furthermore, we have
$$P\sigma (E_{ij})=\sigma P(E_{ij}), \quad 1\leq i\leq j\leq 2,$$
where $E_{ij}$ denotes  the standard basis matrix with $1$ in the $(i,j)$-th entry and with $0$ elsewhere.
Thus $P\sigma=\sigma P$ on $M^u_2(\bfk)$.
Therefore, we have
\vspace{-.1cm}
$$P(\sigma-\id)=(\sigma-\id)P.$$
By
Proposition~\mref{p:endopdrb}, $(R, P, \sigma-\id)$ is
a \pdrba of type II.
\end{exam}
\vspace{-.5cm}
\section{Monomial orders and Gr\"obner-Shirshov bases for operated algebras}
\mlabel{s:nogsb}

In this section, we recall the explicit constructions of free $\Omega$-operated semigroups and free $\Omega$-operated algebras by bracketed words. We then apply these constructions to introduce a monomial order to study Gr\"{o}bner-Shirshov bases for operated algebras.

\subsection{Free operated algebras}
We follow the notions and constructions of free $\Omega$-operated algebras in~\cite{Gop}, a notion that can be traced back to Kurosh~\cite{Ku}.   

\begin{defn}\cite{Gop}
Let $\Omega$ be a set.
\begin{enumerate}
\item
An {\bf $\Omega$-operated semigroup} is a semigroup $S$ equipped with a family of operators $\alpha_\omega: S\to S$ for  $\omega\in\Omega$.
\item
A {\bf morphism  of operated semigroups} from $(S_1,\alpha_\omega)$  to $(S_2,\beta_\omega)$ is a semigroup homomorphism $f:S_1 \to S_2$ such that $f\circ \alpha_\omega= \beta_\omega \circ f$ for $\omega\in\Omega$.
\item
An {\bf $\Omega$-operated $\bfk$-algebra} (resp. {\bf$\Omega$-operated nonunitary $\bfk$-algebra}) is an algebra (resp. nonunitary algebra) $A$ equipped with a family of linear operators $P_\Omega:=\{P_\omega\}_{\omega\in\Omega}$ on $A$.
\item
A {\bf morphism  of operated $\bfk$-algebras} from $(A,P_{A,\Omega})$  to $(B,P_{B,\Omega})$ is an algebra homomorphism $f:A \to B$ such that $f\circ P_{A,\omega}= P_{B,\omega} \circ f$ for  $\omega\in\Omega$.
\end{enumerate}
 \end{defn}

For a set $Y$, let $S(Y)$ denote the free semigroup generated by $Y$, and $M(Y)$ be the free monoid generated by $Y$, both constructed in terms of words from the alphabet set $Y$.

For $\omega\in \Omega$, denote by  $\lfloor Y \rfloor_\omega$  the set of all elements $ \lfloor y \rfloor_\omega, y\in Y$. Set
$$\lfloor Y \rfloor_\Omega:=\sqcup_{\omega\in \Omega} \left \lfloor Y\right \rfloor_\omega.$$
We now construct the free commutative $\Omega$-operated semigroup on a set $Z$ by a direct system
$$\Big\{\tau_{n,n+1}:\frakS_{\Omega,n-1}(Z)\hookrightarrow  \frakS_{\Omega,n}(Z)\Big\}_{n\geq0}$$
 of free semigroups $\frakS_{\Omega,n}(Z)$, where $\tau_{n,n+1}$  is a semigroup homomorphism.
Firstly, we define 
\begin{equation}\mlabel{eq:sz0}
\frakS_{\Omega,0}(Z):=S(Z),\quad
\frakS_{\Omega,1}(Z):= S(Z\sqcup  \lc S(Z)\rc_\Omega).
\end{equation}
 The inclusion into the first component  $ Z\hookrightarrow Z\sqcup \lfloor \frakS_{\Omega,0} \rfloor_\Omega$ induces an injective semigroup homomorphism
$$ \tau_{0,1}: \frakS_{\Omega,0}(Z)  \hookrightarrow \frakS_{\Omega,1}(Z).
$$

For a given $n \geq 1$,   assume by induction that we have defined  $\Omega$-operated semigroups $\frakS_{\Omega,i}(Z)$  with the
 properties that  $\frakS_{\Omega,i}(Z)=S(Z\cup \lc\frakS_{\Omega,i-1}\rc_\Omega)$ and $\tau_{i-1,i}:\frakS_{\Omega,i-1}(Z)\hookrightarrow \frakS_{\Omega,i}(Z)$ for all $1\leq i\leq n$.
Then define
$$\frakS_{\Omega,n+1}(Z):= S(Z\sqcup \lfloor \frakS_{\Omega,n}(Z) \rfloor_\Omega).$$
The homomorphism of semigroups
$\tau_{n-1,n}: \frakS_{\Omega,n-1}(Z) \hookrightarrow \frakS_{\Omega,n}(Z)$ gives an injection 
$$\lc \tau_{n-1,n}\rc_\Omega: \lc \frakS_{\Omega,n-1}(Z)\rc_\Omega \to \lc  \frakS_{\Omega,n}(Z)\rc_\Omega.$$
Then the natural  inclusion
$$
\id_Z\sqcup \lfloor \tau_{n-1,n} \rfloor_\Omega:Z\sqcup \lfloor \frakS_{\Omega,n-1}(Z) \rfloor_\Omega \hookrightarrow Z\sqcup \lfloor \frakS_{\Omega,n}(Z) \rfloor_\Omega
$$
leads to a semigroup  homomorphism
$$	\tau_{n,n+1}: \frakS_{\Omega,n}(Z)= S(Z\sqcup \lfloor  \frakS_{\Omega,n-1}(Z) \rfloor_\Omega) \hookrightarrow  \frakS_{\Omega,n+1}(Z) = S(Z\sqcup \lfloor \frakS_{\Omega,n}(Z) \rfloor_\Omega).$$
This completes the desired direct system. Then by taking the direct limit, we obtain a semigroup
$$  \frakSO(Z):=\bigcup_{n\geq 0}\frakS_{\Omega,n}=\varinjlim \frakSO_{,n}(Z). $$
For a given $\omega\in \Omega$, the sequence of maps 
$$ P_{\omega,n}: \frakSO_{,n}(Z)\longrightarrow  \frakSO_{,n+1}(Z), \quad u\mapsto \left \lfloor u\right \rfloor_{\omega},$$ 
induces by taking direct limit a map $P_{\omega}$  on $\frakSO(Z)$.
Then by~\cite[Corollary~3.6]{Gop}, $\frakSO(Z)$ together with the set $P_\omega:=\{P_\omega\,|\,\omega\in \Omega\}$ forms the free $\Omega$-operated semigroup  on $Z$.

Denote by $\bfk\frakSO(Z)$ the $\bfk$-module with basis $\frakSO(Z)$. Extending the
multiplication on $\frakSO(Z)$ by bilinearity and the set maps $P_\Omega$ by linearity, still denoted by $P_\Omega$, we obtain an $\Omega$-operated algebra $\bfk\frakSO(Z)$.
Then by ~\cite[Corollary~3.7]{Gop},
$\bk\frakSO(Z)$ with $ P_\omega$ becomes the free $\Omega$-operated nonunitary algebra on $Z$.

By a parallel construction, we obtain the free $\Omega$-operated monoid $\frakM_\Omega(Z)$ and the free  $\Omega$-operated unitary algebra $\bfk\frakM_\Omega(Z)$ on a set $Z$. Simply replace the free semigroup $S(Y)$ by the free monoid $M(Y)$ everywhere in the above construction of the free $\Omega$-operated nonunitary  algebra.

\begin{defn} Elements of $\frakSO(Z)$ are called {\bf $\Omega$-bracketed words} or {\bf $\Omega$-operated words}, or simply just
{\bf bracketed words} or {\bf operated words}.  An element $f \in \bfk\frakSO(Z)$ will be
called an {\bf $\Omega$-bracketed polynomial} or {\bf $\Omega$-operated polynomial}, or simply just
{\bf bracketed polynomial} or {\bf operated polynomial},
unless otherwise noted. 
\mlabel{def:bwbp}
\end{defn}
When there is no danger of confusion, we often omit the adjectives ``bracketed'' and ``operated''.

\subsection{Operated polynomial identities}
In this section, we will recall the notions of an operated polynomial identity, and further an operated polynomial identity algebra.

\begin{defn}
\begin{enumerate}
\item
Let $X$ be a set. An element $\phi(x_1,\cdots, x_n)$ in $\bfk\frakSO(X)$ for $x_i\in X$ and $1\leq i\leq n$, is called  an \textbf{operated polynomial identity }(OPI).
\item
Let $\phi(x_1,\cdots, x_n)\in \bfk\frakSO(X)$ be an OPI. We say that an $\Omega$-operated algebra $R$ with a family of linear operators $\{P_\omega\}_{\omega\in\Omega}$ is a {\bf $\phi$-algebra} and that $P_\omega$ is a {\bf $\phi$-operator} for all $\omega\in\Omega$, if ${\phi}(r_1,\dots,r_n)=0$ for all $r_1,\dots,r_n\in R$. An {\bf operated polynomial identity algebra} is any algebra that is also a $\phi$-algebra for some $\phi$.
\item More generally,  for a family of OPIs $\Phi$, we call an  operated algebra $(A,P_\Omega)$ an {\bf $\Phi$-
algebra} if it is an $\phi$-algebra for all $\phi\in\Phi$.
\end{enumerate}
 \mlabel{de:pio}
\end{defn}

\begin{defn}
Given any set $Z$, and a subset $S \subset \bfk\frakSO(Z)$, the {\bf operated ideal $\Id(S)$ of ~$\bfk\frakSO(Z)$ generated by $S$} is the smallest operated ideal containing $S$.
\mlabel{de:repgen}
\end{defn}
Let $\Phi$ be a set of OPIs in $\bfk\frakSO(X)$. Let
$$S_\Phi(Z):=\Big\{\phi(u_1,\cdots,u_n)\,|\, u_1,\cdots,u_n\in\frakSO(Z),\,\phi\in\Phi\Big\}\subset \bfk\frakSO(Z).$$
Denote by $\Id(S_\Phi(Z))$ the operated ideal generated by the set $S_\Phi(Z)$.

\begin{theorem}~\cite{Cohn,QQWZ21} Let $\Phi\subseteq \bfk\frakSO(X)$ be a set of OPIs. Let $Z$ be a set. Then
the quotient  $\bfk\frakSO(Z)/\Id(S_\Phi(Z))$ is the free $\Phi$-algebra on $Z$.
\mlabel{thm:quotient}
\end{theorem}

\begin{exam} Let $X$ be a set and let $\Omega=\left\{D,P\right\}$ be a set consisting of two distinct elements. Let $\lambda\in \bfk$. Let $\Phi_{\rm DRB}$ be the subset of $\bfk\frakSO(X)$ consisting of the following three OPIs:
\begin{enumerate}
   \item $\phi_1(x,y):=\lc xy\rc_D-\lc x\rc_D y - x\lc y\rc_D-\lambda \lc x\rc_D\lc y\rc_D$,
    \item  $\phi_2(x):= \lc\lc x\rc_P\rc_D-x$,
\item $\phi_3(x,y):=\lc x\rc_P \lc y\rc_P-\lc x\lc y\rc_P\rc_P-\lc \lc x\rc_P y\rc_P-\lc xy\rc_P$.
\end{enumerate}
If an operated algebra $\left(A,\{P,D\}\right)$ is an $\Phi_{\rm DRB}$-algebra, then it is just a differential Rota-Baxter algebra of weight $\lambda$\,\mcite{GK3}.
\end{exam}
\begin{defn}
Let $u\in \frakSO(Z)$.
\begin{enumerate}
\item We
 write 
\begin{equation}\mlabel{eq:decom0}
u:=v_1\cdots v_k,
\end{equation}
 where $v_i\in Z\cup \lc \frakSO(Z)\rc_\Omega$ for
$1\leq i\leq k$. We call $k$ the {\bf breadth} of $u$ and denote it
by $|u|$.
\item
By combining adjacent
factors of $u=v_1 \dots v_k$ that belong to $Z$ into a monomial
belonging to $M(Z)$ and by inserting $1 \in M(Z)$ between two adjacent factors
of the form $\lc v \rc_{\omega_i}$, where $v \in \frakSO(Z)$ and $\omega_i\in\Omega$, we may
write $u$ uniquely in the canonical form
\begin{equation}
u=u_0\lc u_1^\ast\rc_{\omega_1} u_1 \cdots \lc u_r^\ast \rc_{\omega_r} u_r, {\rm\ where\ } u_0,\cdots, u_r\in
M(Z),  {\rm\ and\ }u_1^\ast,\cdots, u_r^\ast \in \frakSO(Z). \mlabel{eq:decom1}
\end{equation}
If we write
\begin{equation}
u=v_0\lc v_1^\ast\rc_{P} \cdots \lc v_k^\ast \rc_{P} v_k, \quad\text{\ where\ } v_0,\cdots, v_k\in
M(Z)\cup\lc\frakS_{\Omega}(Z)\rc_{\Omega-\{P\}},  {\rm\ and\ }v_1^\ast,\cdots, v_k^\ast \in \frakSO(Z)\mlabel{eq:decomp}
\end{equation}
we define the {\bf $P$-breadth} of $u$ to be $k$ and denote it by
$|u|_P$.
Further define the {\bf $P$-degree} $\deg_P(u)$ of a monomial
$u$ in $\bfk\frakSO(Z)$ to be the total number of occurrences of
the operator $\lc\ \rc_P$ in the monomial $u$, counted with multiplicity.
\end{enumerate}
\mlabel{def:breadths}
\end{defn}

\subsection{A monomial order on $\frakSO(Z)$}
We now construct a monomial order on $\frakSO(Z)$ with the operator set $\Omega=\{P,D\}$.

\noindent
{\bf Notations.} To simplify notations, in the following sections, we will
use the infix notation $\lc u\rc_P$ (resp. $\lc u\rc_D$) interchangeably with $P(u)$ (resp. $D(u)$) for any $u\in \frakSO(Z)$.

We recall the following notions on orders. 
See \mcite{BCQ,QQWZ21,ZGGS} for further details. 

Let $Z$ be a set with a well order $\leq_Z$. For $u=u_1\cdots u_r\in S(Z)$ with $u_1,\cdots,u_r\in Z$, define $\deg_Z(u)=r$. Also define $\deg_Z(1):=0$.
 Define
$$u\leq_{\dz} v\Leftrightarrow \deg_{Z}(u)\leq \deg_{Z}(v).$$
Then we define the {\bf degree lexicographical order} $\leq_{\dl}$  on $S(Z)$ as follows. For  any $u=u_1\cdots u_r,v=v_1\cdots v_s \in S(Z)$, where $u_1,\cdots,u_r,v_1,\cdots, v_s\in Z$,
$$ u\leq_{\dl}v\Longleftrightarrow \left\{ \begin{array}{l}
u<_{\dz} v,\,\text{or}~ \\[5pt]
u=_{\dz} v ~ \text{and} ~ u_1\cdots u_r\leq_{\lex} v_1\cdots v_r,\end{array}\right.$$
where $\leq_{\lex}$ is the lexicographical order on $S(Z)$.
Then
$\leq_{\dl}$ is a well order on	$S(Z)$\,\cite{BN}.

\begin{defn}
	Let $Z$ be a set.  A {\bf monomial order} on  $\frakSO(Z)$  is a well-order $\leq$ on $\frakSO(Z)$  such that
	$$u< v \Rightarrow q|_u<q|_v,\quad u,v\in\frakSO(Z)\text{ and } q\in \frakSO^{\star}(Z).$$
\end{defn}
The following are further notions on monomial orders. 
\begin{defn} Let $Z$ be a set and let $\leq$ be a monomial order on $\frakSO(Z)$.
\begin{enumerate}
\item
For every bracketed polynomial $f\in \bk\frakSO(Z) $, denote by
$\bar{f}$  the leading bracketed monomial of $f$. If the coefficient
of $\bar{f}$ is $1$, then we call $f$ {\bf monic}.
\item
Set  $Z^\star:=Z\cup  \{\star \}$, where $\star\notin Z$. Define a {\bf $\star$-bracketed word} to be a bracketed word in $\frakSO(Z^\star)$  with $\star$ occurring only once. Denote by $\frakSO^{\star}(Z)$  the subset of $\frakSO(Z^\star)$ consisting of $\star$-bracketed words.   Let $q\in\frakSO^\star(Z)$ and let $u\in\frakSO(Z)$. Define
  $$ q|_{u}:=q|_{\star\mapsto u }$$
to be the bracketed word in $\frakSO(Z)$ obtained by replacing the symbol $\star$ in $q$ by $u$.
\item More generally, for $q\in\frakSO^{\star}(Z)$ and $s=\sum_ic_iu_i\in \bfk \frakSO(Z)$  with $c_i\in\bk$ and $u_i\in\frakSO(Z)$, we define
	$$q|_s:=\sum_ic_iq|_{u_i}.$$
\item
Let $Z^{\star_1,\star_2}:=Z\cup\{\star_1,\star_2\}$.  Similarly, we define a {\bf $(\star_1, \star_2)$-bracketed word} to be a bracketed word in  $\frakSO(Z^{\star_1,\star_2})$ with exactly one occurrence of $\star_1$ and exactly one occurrence of $\star_2$, each counted with multiplicity. The set of $(\star_1,\star_2)$-bracketed words is denoted by $\frakSO^{\star_1,\star_2}(Z)$.
\item
Let $q\in \frakSO^{\star_1,\star_2}(Z)$ and $u_1,u_2\in\frakSO(Z)$. Then we define
$$
q|_{u_1,\, u_2}:= u|_{\star_1\mapsto u_1,\,\star_2\mapsto u_2},
$$
to be the bracketed word in $\frakSO(Z)$ obtained by replacing the symbol $\star_1$ in $q$ by $u_1$, and
replacing the symbol $\star_2$ in $q$ by $u_2$.
\end{enumerate}
\end{defn}

With the above notations,  the operated ideal $\Id(S)$ generated by a subset $S \subseteq \bfk\frakSO{Z}$ can be given by
\begin{equation}\hspace{10pt}\Id(S) = \left\{\, \sum_{i=1}^k c_i q_i|_{s_i} \medmid k\geq 1
{\rm\ and\ } c_i\in \bfk,
q_i\in \frakSO^{\star}(Z), s_i\in S {\rm\ for\ } 1\leq i\leq k\,\right\}.
\mlabel{eq:repgen}
\end{equation}

There are various monomial orders on the operated semigroup $\frakSO(Z)$. See \mcite{LQQZ23,QQWZ21,QQZ21,ZGGS} for more details. For our purpose, we will construct a monomial order on $\frakS_\Omega(Z)$, similar to $\leq_{PD}$ given in \mcite{LQQZ23}.
\begin{defn}
Let $u$ and $v$ be elements in $\frakSO(Z)$.
\begin{enumerate}
\item Define $$u\leq_\p v\Longleftrightarrow \deg_P(u)\leq \deg_P(v),$$ where the {\bf $P$-degree} $\deg_P(u)$ of $u$ is the number of occurrences of  $P=\lc\ \rc_P$ in $u$.
\item Similarly define $$u\leq_{\rm d} v\Longleftrightarrow \deg_D(u)\leq \deg_D(v),$$ where the {\bf $D$-degree} $\deg_D(u)$ of $u$ is the number of occurrences of  $D=\lc \ \rc_D$ in $u$.
\item For an appearance of the bracket $D=\lc\ \rc_D$ in $u$, if the bracket $D$ contains $k$ elements of $Z$,  we say that  $D$ has ED-degree $k-1$.  The {\bf ED-degree} of $u$, denoted by $\deg_{\ed}(u)$,  is defined to be the sum of the ED-degrees for all the brackets $D$ in $u$.  Then define
$$u\leq_{\ed} v \Longleftrightarrow \deg_{\ed}(u)\leq \deg_{\ed}(v).$$
For instance,  let $x,y\in Z$. Then 
$$\deg_{\ed}(D(xy))=1,\  \deg_{\ed}(D(x)D(y))=\deg_{\ed}(xD(y))=\deg_{\ed}(D(x)y)=0,$$ and so $ D(x)D(y)<_{\ed} D(xy)$, $D(x)y<_{\ed}D(xy) $ and $xD(y)<_{\ed} D(xy)$.
\item~\cite[Definition~1.10]{LQQZ23} Let $u\in\frakSO(Z)$, and let $u_1,\cdots,u_n\in Z$ be all the elements occurring in $u$. For a given bracket $P$ in $u$, if there are $ k$ elements $u_{i_1},\cdots,u_{i_k}$ contained in $P$  for $ i_1,\cdots,i_k\in\{1,\cdots,n\}$, the {\bf GP-degree} of $P$ is
 defined to be $n-k$. Denote by $\deg_{\rm gp}(u)$ the sum of the GP-degree of all brackets $P$ in $u$.
Then we define
$$u\leq_{\rm gp} v\Longleftrightarrow \deg_{\rm gp}(u)\leq \deg_{\rm gp}(v).$$
For example,  for $x,y\in Z$, we have $$\deg_{\rm gp}(P(x)P(y))=2,\, \deg_{\rm gp}(P(xP(y)))=\deg_{\rm gp}(P(P(x)y))=1\,\text{and}\, \deg_{\rm gp}(P(xy))=0.$$ Thus $P(xy)<_{\rm gp} P(P(x)y)=_{\rm gp}P(xP(y))<_{\rm gp} P(x)P(y)$.
\item
According to the filtration $\frakSO(Z)=\bigcup_{n\geq 0}\frakS_{\Omega,n}$,  we now define a preorder $\leq_{\operatorname{Dlex}}$ on $\frakSO(Z)$ as the direct limit of a sequence of well orders $\leq_{\operatorname{Dlex}_n}$ on $\frakS_{\Omega,n}, n\geq 0,$
which are constructed by the following recursion.
\begin{enumerate}
\item When $n=0$, then we have $\frakS_{\Omega,n}=S(Z)$. Let $u,v\in \frakS_{\Omega,0}(Z)$. Define
	$$u\leq_{\operatorname{Dlex}_0} v \Longleftrightarrow u \leq_{\mathrm{dlex}}v.$$
\item For a given $n\geq 0$, assume that a well order $\leq_{\operatorname{Dlex}_i}$  on $\frakS_{\Omega,i}(Z)$
has been constructed for all $0\leq i\leq n$.
The  well order $\leq_{\operatorname{Dlex}_n}$ on $\frakS_{\Omega,n}(Z)$ induces a well order   on the set $P(\frakS_{\Omega,n}(Z))$  (resp. $D(\frakS_{\Omega,n}(Z))$), by defining
$P(u)\leq P(v)$  (resp.  $D(u)\leq D(v)$) when $u\leq_{\operatorname{Dlex}_n} v$ for $u,v\in\frakS_{\Omega,n}(Z)$.
Define
\begin{equation}u<v<w,\quad \forall u\in Z, v\in P(\frakS_{\Omega,n}(Z)), w\in D(\frakS_{\Omega,n}(Z)).
\mlabel{eq:uvw}
\end{equation}
For example, $P(x)<_{\rm{Dlex}_1} D(x)$ and $P(D(x))<_{\rm Dlex_2} D(P(x))$.

Then by \cite[P214, Chapter VII, \S1, Example~4]{KM76}, we obtain a  well order  on  $Z\sqcup P(\frakS_{\Omega,n}(Z)) \sqcup D(\frakS_{\Omega,n}(Z))$.
Let $\leq_{\operatorname{Dlex}_{n+1}}$ be  the  degree lexicographic order on   $\frakS_{\Omega,n+1}(Z)=S(Z\sqcup P(\frakS_{\Omega,n}(Z)) \sqcup D(\frakS_{\Omega,n}(Z)))$  induced by the well order given in Eq.~\meqref{eq:uvw}.
\end{enumerate}
Taking the direct limit, we obtain a preorder on $\frakS(Z)$
\begin{equation}\mlabel{eq:dlex}
\leq_{\operatorname{Dlex}}:=\varinjlim \leq_{\operatorname{Dlex}_n}.
\end{equation}
By a direct computation,  $\leq_{\operatorname{Dlex}}$ is a linear order. But it is not a well order. Here is a counterexample: Consider the sequence $P^n(xP(y)), n\geq 0$. Then we have $P^n(xP(y))>_{\rm Dlex} P^{n+1}(xP(y))$, leading to an infinite decreasing sequence. 
\mlabel{it:dlex}
\item
We finally define an order $\leq_{\zdp}$ on $\frakSO(Z)$ by combining the above orders
\begin{equation}u \leq_{\zdp} v  \Longleftrightarrow \left\{
\begin{array}{lcl}
u<_{\dz} v,\,\text{or}\\
u=_{\dz} v, \text{and}\;u <_{\ed} v, \text{or}\\
u=_{\dz} v, u =_{\ed} v, \text{and}\; u <_{\rm d} v,\text{or} \\
u=_{\dz} v, u =_{\ed} v, u =_{\rm d} v, \text{and}\;  u <_{\rm gp}v, \text{or} \\
u=_{\dz} v, u =_{\ed} v, u =_{\rm d} v,  u =_{\rm gp}v,\text{and}\;u <_{\p}v, \text{or}\\
u=_{\dz} v, u =_{\ed} v, u =_{\rm d} v,  u =_{\rm gp}v, u =_{\p}v, \text{and}\;  u \leq_{\dlex}v.\\
\end{array}
 \right.
\mlabel{eq:dp}
\end{equation}
\end{enumerate}
\mlabel{def:preorder}
\end{defn}

To show that $\leq_{\zdp} $ is a well order,
we  first recall the following notions and results of preorders.
\begin{defn}\cite[Definition~5.3]{ZGGS}
\begin{enumerate}
\item Let $k\geq 1$ and let $\leq_{\alpha_i}, 1\leq i\leq k,$ be preorders on a set $Y$. For all $u,v\in Y$,  we define recursively
\begin{equation}
u\leq_{\alpha_1,\cdots,\alpha_k} v \Longleftrightarrow
\left\{\begin{array}{ll} u<_{\alpha_1} v, \text{or } \\[5pt]
u=_{\alpha_1} v \text{ and } u\leq_{\beta} v, \end{array} \right. \mlabel{eq:compord}
\end{equation}
where $\leq_{\beta}:=\leq_{\alpha_2,\cdots,\alpha_k}$ is well-defined by the induction hypothesis.
\item
Let $k\geq 1$ and let $(Y_i,\leq_{Y_i}), 1\leq i\leq k$, be partially ordered sets. Recall that the {\bf lexicographical order} $\leq_{\clex}$ on the cartesian  product $Y_1  \times Y_2\times\cdots \times Y_k$ is defined recursively by
\begin{eqnarray}(x_1,x_2,\cdots,x_k) &\leq_{\clex}&(y_1,y_2, \cdots,y_k) \nonumber\\&\Longleftrightarrow&
\left\{\begin{array}{ll} x_1<_{Y_1}y_1,\text{or}~ \\[5pt]
x_1= y_1 ~\text{and}~ (x_2,\cdots,x_k)\leq_{\clex}(y_2, \cdots,y_k),
\end{array}\right.\nonumber
\end{eqnarray}
where $(x_2,\cdots,x_k)\leq_{\clex} (y_2,\cdots,y_k)$ is well-defined by the induction hypothesis.
\end{enumerate}
\end{defn}
		
\begin{lemma}
\begin{enumerate}
\item\cite[Lemma~1.7]{QQZ21}
Let $k\geq 2$. Let $\leq_{\alpha_1},\cdots, \leq_{\alpha_{k-1}}$ be pre-linear orders on $Z$, and let $\leq_{\alpha_k}$ be a linear order on $Z$. Then the relation $\leq_{\alpha_1,\cdots,\alpha_k}$ is a linear order on $Z$. \mlabel{it:well1}
\item
{\bf{\cite[P89,~Theorem~1.13]{Ha}}} Let $\leq_{Y_i}$ be a well order on $Y_i$ with $1\leq i\leq k$ and  $k\geq 1$. Then the lexicographical order ~$\leq_{\clex}$ is a well order on the Cartesian product $Y_1 \times Y_2\times\cdots \times Y_k$.\mlabel{it:well2}
\end{enumerate}
\mlabel{le:wellord}
\end{lemma}

\begin{prop}
The order $\leq_{\zdp} $ is a well order on $\frakSO(Z)$.
\mlabel{prop:wellord}
\end{prop}
\begin{proof}
The orders $\leq_{\dz}, \leq_{\ed}, \leq_{\rm d},  \leq_{\rm gp}$ and $\leq_{\p}$ are pre-linear orders since they take values in $\NN$. By Definition~\mref{def:preorder}~(\mref{it:dlex}), the order $\leq_{\rm Dlex}$ is a linear order.  Then by Lemma~\mref{le:wellord}~(\mref{it:well1}),
$\leq_{\zdp}$ is a linear order. We next show that $\leq_{\zdp}$ satisfies the descending chain condition.
Suppose that $u_1\geq_{\zdp} u_2\geq_{\zdp}\cdots$. Since $\leq_{\rm d}$ and $\leq_{\p}$ satisfy the descending chain condition, there are natural numbers $N_1, N_2\geq 0$ and $k_1,k_2\geq 0$ such that
$$
\deg_{D}(u_{N_1})=\deg_{D}(u_{N_1+1})=\cdots=k_1
$$
and
$$\deg_{P}(u_{N_2})=\deg_{D}(u_{N_2+1})=\cdots=k_2. $$
Take $N:=\max\{N_1,N_2\}$ and let $k:=k_1+k_2$. Then $u_i\in \frakS_{\Omega,k}(Z)$ for all $i\geq N$. Since the restriction of $\leq_{\zdp}$ to $\frakS_{\Omega,k}(Z)$ is just the well order $\leq_{\dlex_k}$, the descendant chain
$u_{N}\geq_{\zdp} u_{N+1}\geq _{\zdp}\cdots$ means the descendant chain $u_{N}\geq_{\dlex_k} u_{N+1}\geq _{\dlex_k}\cdots$ and hence stabilizes after finite steps. 
\end{proof}

\begin{defn}
 A well order $\leq_{\operatorname{\alpha}}$ on $\frakSO(Z)$ is called {\bf bracket compatible} (resp. {\bf left multiplication compatible}, resp. {\bf right multiplication  compatible}) if
$$u\leq_{\operatorname{\alpha}}v \implies \lfloor u\rfloor_\omega \leq_{\operatorname{\alpha}}\lfloor v \rfloor_\omega\, (\text{resp.}\, wu \leq_{\operatorname{\alpha}}wv, \text{resp.}\, uw\leq_{\operatorname{\alpha}}vw), \, \forall u, v, w \in \frakSO(Z), \omega\in\Omega.$$
\end{defn}

 \begin{lemma}\mlabel{lem:comp}    \cite[Lemma~5.11]{ZGGS} A well order $\leq_{\operatorname{\alpha}}$ on $\frakSO(Z)$ is a monomial order if and only if $\leq_{\operatorname{\alpha}}$ is bracket compatible, left multiplication compatible and right multiplication compatible.
 \end{lemma}

\begin{prop}
 The  order $\leq_{\zdp}$ is a monomial order on $\frakSO(Z)$.
\mlabel{prop:monom}
\end{prop}

 \begin{proof}
By Proposition~\mref{prop:wellord}, $\leq_{\zdp}$ is a well order. According to Lemma~\mref{lem:comp}, it suffices to verify that $\leq_{\zdp}$ is bracket compatible, left multiplication compatible and right multiplication compatible.  That is, for  $u\leq_{\zdp} v$ with $u,v\in\frakSO(Z)$,
$$D(u)\leq_{\zdp} D(v),\  P(u)\leq_{\zdp} P(v),\  wu\leq_{\zdp} wv,\  uw\leq_{\zdp} vw,\,\forall w\in\frakSO(Z).$$
We only prove the first inequality, since the others are similarly verified.
Let $u\leq_{\zdp} v$. Then we distinguish six cases, based on the definition of $\leq_{\zdp}$ in Eq.~\meqref{eq:dp}.
\smallskip

\noindent
{\bf Case 1}: Let $u<_{\dz} v$. Then $\deg_Z(u)<\deg_Z(v)$. So $$\deg_Z(D(u))=\deg_Z(u)<\deg_Z(v)=\deg_Z(D(v),$$ 
and thus $D(u)<_{\dz} D(v)$, proving $D(u)<_{\zdp} D(v)$.
\smallskip

\noindent
{\bf Case 2}: Let $u=_{\dz} v$, and $u <_{\ed} v$. Then $\deg_Z(u)=\deg_Z(v)$ and $\deg_{\ed}(u)<\deg_{\ed}(v)$. Since $$\deg_{\ed}(D(u))=\deg_{\ed}(u)+\deg_Z(u)-1, \quad \deg_{\ed}(D(v))=\deg_{\ed}(v)+\deg_Z(v)-1,$$ 
we have
$\deg_{\ed}(D(u))<\deg_{\ed}(D(v))$, and so $D(u) <_{\ed} D(v)$, giving $D(u) <_{\zdp} D(v)$.
\smallskip

\noindent
{\bf Case 3}: Let $u=_{\dz} v, u =_{\ed} v$, and $u <_{\rm d} v$. Then $$\deg_Z(u)=\deg_Z(v),\ \deg_{\ed}(u)=\deg_{\ed}(v),\ \deg_{D}(u)<\deg_{D}(v).$$ 
Hence $D(u)=_{\dz} D(v)$. By $$\deg_{\ed}(D(u))=\deg_{\ed}(u)+\deg_Z(u)-1=\deg_{\ed}(v)+\deg_Z(v)-1=\deg_{\ed}(D(v)),$$
we obtain $D(u)=_{\ed} D(v)$.  According to $\deg_{D}(D(u))<\deg_{D}(D(v))$, we have $D(u)<_{\rm d} D(v)$, and so $D(u)\leq_{\zdp} D(v)$.
\smallskip

\noindent
{\bf Case 4}: Let $u=_{\dz} v, u =_{\ed} v, u =_{\rm d} v$, and  $u <_{\rm gp} v$. Then 
$$\deg_Z(u)=\deg_Z(v),\ \deg_{\ed}(u)=\deg_{\ed}(v), \ \deg_{D}(u)=\deg_{D}(v), \ \deg_{\rm gp}(u)<\deg_{\rm gp}(v).$$ 
So
$$\deg_Z(D(u))=\deg_Z(D(v)),\ \deg_{\ed}(D(u))=\deg_{\ed}(D(v)), $$ $$\deg_{D}(D(u))=\deg_{D}(D(v)), \ \deg_{\rm gp}(D(u))<\deg_{\rm gp}(D(v)),$$ 
proving $D(u)<_{\zdp} D(v)$.
\smallskip

\noindent
{\bf Case 5}: Let $u=_{\dz} v, u =_{\ed} v, u =_{\rm d} v,  u =_{\rm gp} v$, and  $u <_{\p}v$.
Thus, we get 
$$\deg_Z(u)=\deg_Z(v),\ \deg_{\ed}(u)=\deg_{\ed}(v),\ \deg_{D}(u)=\deg_{D}(v),$$
$$\deg_{\rm gp}(u)=\deg_{\rm gp}(v),\ \deg_P(u)<\deg_P(v).$$ 
Hence, 
$$\deg_Z(D(u))=\deg_Z(D(v)),\ \deg_{\ed}(D(u))=\deg_{\ed}(D(v)),\ \deg_{D}(D(u))=\deg_{D}(D(v)),$$
$$\deg_{\rm gp}(D(u))=\deg_{\rm gp}(D(v)),\ \deg_P(D(u))<\deg_P(D(v)),$$ leading to $D(u)<_{\zdp} D(v)$.
\smallskip

\noindent
{\bf Case 6}: Let $u=_{\dz} v, u =_{\ed} v, u =_{\rm d} v,  u =_{\rm gp}v, u =_{\p}v$, and  $u \leq_{\dlex} v$.
Then 
$$\deg_Z(u)=\deg_Z(v),\ \deg_{\ed}(u)=\deg_{\ed}(v),\ \deg_{D}(u)=\deg_{D}(v),$$
$$\deg_{\rm gp}(u)=\deg_{\rm gp}(v),\ \deg_P(u)=\deg_P(v).$$ 
So we have
$$\deg_Z(D(u))=\deg_Z(D(v)),\ \deg_{\ed}(D(u))=\deg_{\ed}(D(v)),\ \deg_{D}(D(u))=\deg_{D}(D(v)),$$
$$\deg_{\rm gp}(D(u))=\deg_{\rm gp}(D(v)), \ \deg_P(D(u))=\deg_P(D(v)).$$ By $u \leq_{\dlex} v$,  there exists $n\geq 0$ such that $u,v\in \frakS_{\Omega,n}(Z)$. Then $u\leq_{\dlex,n} v$, and so $D(u)\leq _{\dlex,n+1} D(v)$ by the definition of $\leq_{\dlex}$, proving $D(u)<_{\zdp} D(v)$.

In summary, the first inequality is proved. 
\end{proof}

\subsection{Gr\"{o}bner-Shirshov bases}
We first recall some notations and then give the Composition-Diamond Lemma for  associative algebras with multiple linear operators  for later applications~\cite{BCQ}. 

The Composition-Diamond lemma for  Rota-Baxter algebras has been obtained in~\cite{BCD}. See ~\cite{BC,BCC,BCL,ZGGS} for other algebras and further references.

\begin{defn}
Let $\leq$ be a monomial order on $\frakSO(Z)$. Let $f, g \in \bfk\frakSO(Z)$ be two monic bracketed polynomials. Let $\bar{f}$ be the leading bracketed monomial of $f$, and let $|\bar{f}|$ be its breadth.
If there exist $\mu,\nu, w\in \frakSO(Z)$ such that $w=\bar{f}\mu=\nu\bar{g}$ with $\max\{\bre{\bar{f}},\bre{\bar{g}}\}<\bre{w}< \bre{\bar{f}}+\bre{\bar{g}}$, we call the bracketed polynomial
$$(f,g)^{\mu,\nu}_w:=f\mu-\nu g$$
the {\bf intersection composition of $f$ and $g$ with respect to $(\mu,\nu)$}.
If there exist $q\in \frakSO^{\star}(Z)$ and
$w \in \frakSO(Z)$ such that $w =\bar{f}=q|_{\bar{g}}$,  we call the bracketed polynomial
$$(f,g)^{q}_w:=f-q|_{g}$$ the {\bf including
composition of $f$ and $g$ with respect to $q$}. In either case, we call $w$ an {\bf ambiguity} of the corresponding intersection (or including) composition.
\end{defn}

\begin{defn}\mlabel{de:trivial}
Let $w \in \frakSO(Z)$, $S$ be a set of monic bracketed polynomials in $\bfk\frakSO(Z)$, and $\Id(S)$ be the bracketed ideal generated by $S$ as given  by Eq.\,(\mref{eq:repgen}).  An operated polynomial $f \in \bfk\frakSO(Z)$ is called  {\bf trivial modulo $(S, w)$} if it can be written in the form
$$f=\sum_ic_iq_i|_{s_i},$$
where $0 \ne c_i\in \bfk$,  $q_i \in \frakSO^{\star}(Z)$, $s_i\in S$ and $q_i|_{\overline{s_i}}< w$. In this case, we also write
$$f\equiv0\; \modk (S,w).$$
\end{defn}

\begin{defn}\mlabel{de:GS}
A set $S\subseteq \bfk\frakSO(Z)$ of monic bracketed polynomials is called a {\bf Gr\"{o}bner-Shirshov basis with respect to $\leq$} if, for all pairs $f,g\in S$, every intersection composition of the form $(f,g)_w^{\mu,\nu}$ is trivial modulo $(S, w)$, and every including composition of the form $(f,g)_w^q$ is trivial modulo $(S, w)$.
\end{defn}

The following Composition-Diamond Lemma is an efficient instrument for constructing a basis of free \pdrbas.

\begin{theorem}\mlabel{thm:CDL} {\rm (Composition-Diamond Lemma \cite{BCQ,GSZ,ZGGS})} Let $Z$ be a set and let $\leq$ be a monomial order on $\frakSO(Z)$.
Let $S$ be a set of monic bracketed polynomials in $\bfk\frakSO(Z)$. Let 
\begin{equation*} 
\eta\!: \bfk\frakSO(Z) \to \bfk\frakSO(Z)/\Id(S)
\end{equation*} 
be the canonical homomorphism of $\bfk$-modules.  Then the following statements are equivalent.
\begin{enumerate}
\item $S $ is a Gr\"{o}bner-Shirshov basis in $\bfk\frakSO(Z)$.
\mlabel{it:cd1}
\item If $f\in Id(S)$ then $\overline{f}=u|_{\overline{s}}$ for some $s\in S$ and $u \in \frakSstar(Z)$ \mlabel{it:cd2}
\item
Let
$$\Irr(S):=\frakSO(Z)\backslash\{q|_{\bar{s}}\,|\,s\in S,q\in\frakSO^\star(Z)\}.$$
Then, $\bfk\frakSO(Z)=\bfk \Irr(S)\oplus \Id(S)$
and $\eta(\Irr(S))$ is a $\bfk$-basis of $\bfk\frakSO(Z)/\Id(S)$ \mlabel{it:cd3}
\end{enumerate}
\end{theorem}

We also recall the following lemma for later application. 
\begin{lemma}\cite[Lemma~3.5]{LQQZ23}
Let $\phi:=\phi(x_1,\dots,x_n)$ and $\psi:=\psi(y_1,\dots,y_m)$ be two OPIs.  Let $Z$ be a set, and let $u_1, \dots, u_n, v_1, \dots, v_m\in \frakSO(Z)$ be given.
If there exist $i$ for $1 \leq i \leq n$, and $r \in \frakSstar(Z)$, such that
$u_i = \left.r\right|_{\bar{\psi}}$, then the including composition $(\phi, \psi)_w^q= \phi - q|_{\psi}$ with $w=\bar \phi$ and $q=\bar{\phi}(u_1,\cdots,u_{i-1},r,u_{i+1},\cdots,u_n)\in \frakSstar(Z)$  is trivial modulo $(S_{\{\phi,\psi\}}(Z),w)$. In this case, this composition is called a {\bf complete including composition}.
\mlabel{lem:incl}
\end{lemma}

\section{Gr\"{o}bner-Shirshov bases for para-differential Rota-Baxter algebras}
\mlabel{s:gsbqdrb}

In this section, we will provide Gr\"{o}bner-Shirshov bases for the \pdrbas of types I, II and III. Let $d$ be a differential operator of weight $\lambda$ and $P$ be a Rota-Baxter oeprator of weight $\lambda$ on an algebra $R$. They are related by the operator identity:
\begin{equation}
	dP-Pd-bP, \quad \bz\in\bfk.
	\mlabel{eq:dpmix}
\end{equation}
Recall from Definition\,\mref{de:pdrbtype} that the triple $(R,d,P)$ is called a \pdrba of {\bf type I} (resp. {\bf type II}, resp. {\bf type III}) if the operator identity in Eq.~\meqref{eq:dpmix} holds with
$$
 \lambda=0, b=0, \quad \text{(resp.
}  \lambda\neq 0, b=0, \text{ resp.  }  \lambda=0, b\neq 0).
$$

We will give their defining relations, regarded as elements in $\bfk\frakSO(\{x,y\})$, and then show that they give  Gr\"obner-Shirshov bases for the three types of \pdrbas defined. 

For the \pdrbas of type I, namely $\lambda= 0$ and $b=0$, the defining operator identities are 
\begin{enumerate}
	\item   $\phi_{\rmI,1}(x,y) := P(x) P(y)-P(x P(y))-P(P(x) y)$,
	\item   $\phi_{\rmI,2}(x,y) :=D(xy)- D(x)y - xD(y)$,
	\item  $\phi_{\rmI,3}(x) := DP(x)-PD(x)$.
\end{enumerate}
Let $X$ be a set.  Let $\Phi:=\{\phi_k:=\phi_k(x_1,\cdots,x_n)\,|\,k\geq 1, x_1,\cdots,x_n\in X\}\subseteq \bfk\frakSO(X)$ be a family of OPIs. For any given set $Z$, we denote
$$S_\Phi(Z):=\{\phi_k(u_1,\cdots, u_n)\,|\, \phi_k\in\Phi,u_1,\cdots,u_n\in\bfk\frakSO(Z)\}.$$
In particular,  we let
\begin{equation}
 S_{\Phi_\rmI}(Z):=\{\phi_{\rmI,1}(u,v),\,\phi_{\rmI,2}(u,v),\phi_{\rmI,3}(u)\,|\,u,v\in\bfk\frakSO(Z)\}. 
 \mlabel{eq:phi1}
 \end{equation}

Similarly, for \pdrbas of type II, we have $\lambda\neq 0$, $b=0$, and have the defining relations $\Phi_{\rm II}$ comprised of the following OPIs.
\begin{enumerate}
	\item   $\phi_{\rmII,1}(x,y) := P(x) P(y)-P(x P(y))-P(P(x) y)- \lambda P(x y)$,
	\item   $\phi_{\rmII,2}(x,y) :=D(xy)- D(x)y - xD(y)-\lambda D(x)D(y) $,
	\item  $\phi_{\rmIII,3}(x) := DP(x)-PD(x)$.
\end{enumerate}
Consider the set
\begin{equation}
	S_{\Phi_{\rm II}}(Z):=\{\phi_{\rmII,1}(u,v),\,\phi_{\rmII,2}(u,v),\phi_{\rmIII,3}(u)\,|\,u,v\in\bfk\frakSO(Z)\}.
	\mlabel{eq:phi2}
\end{equation}

Finally, for \pdrbas of type III, we have  $\lambda= 0$, $b\neq 0$. The set $\Phi_{\rm III}$ of defining relations consists of the OPIs: 
\begin{enumerate}
	\item   $\phi_{\rmIII,1}(x,y) := P(x) P(y)-P(x P(y))-P(P(x) y)$,
	\item   $\phi_{\rmIII,2}(x,y) :=D(xy)- D(x)y - xD(y)$,
	\item  $\phi_{\rmIII,3}(x) = DP(x)-PD(x)-bP(x)$.
\end{enumerate}
Consider the set
\begin{equation}
	S_{\Phi_{\rm III}}(Z):=\{\phi_{\rmIII,1}(u,v),\,\phi_{\rmIII,2}(u,v),\phi_{\rmIII,3}(u)\,|\,u,v\in\bfk\frakSO(Z)\}. 
	\mlabel{eq:phi3}
\end{equation}

\begin{theorem} Let $\rmT\in\{\rmI,\rmII,\rmIII\}$. The set $S_{\Phi_{\rmT}}(Z)$ is a Gr\"{o}bner-Shirshov basis in $\bfk\frakSO(Z)$ with respect to the monomial order $\leq_{\zdp}$ in Eq.\,\meqref{eq:dp}.
	\mlabel{thm:phigsb}
\end{theorem}

\begin{proof}
We first verify the case of type I. 
Let $\phi_i:=\phi_{\rmI,i}(u,v),\phi_j:=\phi_{\rmI,j}(u,v)\in S_{\Phi_{\rmI}}(Z), i, j\in \{1,2,3\}$. Denote by $i \wedge j$ the OPI composition of $\phi_i$ and $\phi_j$, which means that either $\phi_i$ lies on the left and $\phi_j$ lies on the right of the intersection composition, or $\phi_j$ is included in $\phi_i$ for the inclusion composition. The
ambiguities $t$ of all compositions are listed below: for all $u,v,w\in \bfk\frakSO(Z)$ and $r\in\frakSO^\star(Z)$,
	\begin{itemize}
		\item [$1 \wedge 1$] \quad $\underline{P(u)P(v)P(w)}$, \quad $\uwave{P\left(r|_{P(u)P(v)} \right)P\left(w\right)}$, \quad  $P\left(u\right) P\left(r|_{P(v)P(w)}\right)$.
		\item [$1 \wedge 2$] \quad $P\left(r|_{D(uv)}\right)P(w)$, \quad $P(u)P\left(r|_{D(vw)}\right)$.
		\item[$1 \wedge 3$] \quad  $P\left(r|_{DP(u)}\right)P(v)$, \quad $P(u)P\left(r|_{DP(v)}\right)$.
		\item [$2 \wedge 1$] \quad $\uwave{D\left(r|_{P(u)P(v)}w\right)}$, \quad $D\left(ur|_{P(v)P(w)}\right)$, \quad $\underline{D(P(u)P(v))}$.
		\item [$2 \wedge 2$] \quad $D(r|_{D(uv)}w)$, \quad$D(ur|_{D(vw)})$.
		\item[$2\wedge 3$]\quad $D(r|_{DP(u)}v)$, \quad$D(ur|_{DP(v)})$.
		\item[$3\wedge 1$]\quad $DP\left(r|_{P(u)P(v)}\right)$.
		\item[$3\wedge 2$]\quad $DP\left(r|_{D(uv)}\right)$.
		\item[$3\wedge 3$]\quad  $DP\left(r|_{DP(u)}\right).$
	\end{itemize}

We note that all the compositions except the two underlined ones are from including compositions for which Lemma~\mref{lem:incl} applies. 
We give details of the two underwaved cases as illustrations. 

Consider the first underwaved term $P\left(r|_{P(u)P(v)} \right)P\left(w\right)$.  
This is the case in Lemma~\mref{lem:incl} by first taking
$$\phi:= \phi_{\mathrm{I},1}(u_1, u_2)=P(u_1)P(u_2)-P(u_1P(u_2))-P(P(u_1)u_2), \quad \bar{\phi}=P(u_1)P(u_2),$$
$$\psi:= \phi_{\mathrm{I},1}(v_1, v_2)=P(v_1)P(v_2)-P(v_1P(v_2))-P(P(v_1)v_2), \quad \bar{\psi}=P(v_1)P(v_2),$$ 
and then taking 
$i=1, v_1=u, v_2=v, u_1 = r|_{P(v_1)P(v_2)}=r|_{P(u)P(v)}, u_2=w, q=P(r)P(u_2)=P(r)P(w)$.
Then 
$$
\bar{\phi} = P(u_1)P(u_2)= P\!\left(r|_{P(u)P(v)}\right) P(w) = q|_{\bar{\psi}}.
$$
Now by Lemma~\mref{lem:incl}, $\phi - q|_\psi$ is trivial modulo $(S_{\Phi_{\mathrm{I}}}(Z), t)$.

Next consider the second underwaved term $D\left(r|_{P(u)P(v)}w\right)$. 
Here in Lemma~\mref{lem:incl} we first take 
$$\phi:=\phi_{\rmI,2}(u_1,u_2), \quad \bar{\phi}=D(u_1u_2),$$ 
$$\psi:=\phi_{\rmI,1}(v_1,v_2), \quad \bar{\psi}=P(v_1)P(v_2),$$
and then take $i=1, u_1=P(u)P(v), u_2=w, v_1=u, v_2=v, q=D(rw)$.
Then we have 
$$\bar{\phi}=D\left(r|_{P(v_1)P(v_2)}u_2\right)=D(r|_{P(u)P(v)}w)=q|_{P(u)P(v)}=q|_{\bar{\psi}}.$$ 
Now by Lemma~\mref{lem:incl}, we have that $\phi-q|_\psi$ is trivial modulo  $(S_{\Phi_{\mathrm{I}}}(Z), t)$.

After this analysis, we are left to check the two underlined cases. 
	
First consider the underlined ambiguity $t=P(u)P(v)P(w)$.
	Then we compute the composition
\begin{align*}
&		\phi_1(u,v)P(w)-P(u)\phi_1(v,w) \\
=& -P(uP(v))P(w)-P(P(u)v)P(w)+P(u)P(vP(w))+P(u)P(P(v)w)\\
 =&-\phi_1(uP(v),w) -P(uP(v)P(w))-P(P(uP(v))w)-\phi_1(P(u)v,w)-P(P(u)vP(w))-P(P(P(u)v)w)\\
& +\phi_1(u,vP(w))+P(uP(vP(w)) +P(P(u)vP(w))+\phi_1(u,P(v)w) +P(uP(P(v)w)) + P(P(u)P(v)w)\\
=& -\phi_1(uP(v),w)-\phi_1(P(u)v,w)+\phi_1(u,vP(w))+\phi_1(u,P(v)w)+P(\phi_1(u,v)w)-P(u\phi_1(v,w)).
	\end{align*}
For each of the terms, we get
	\begin{eqnarray*}
		&\overline{\phi_1(uP(v),w)}=P(u P(v))P(w)<_{\zdp} P(u)P(v)P(w),\\
		&\overline{\phi_1(P(u)v,w)}=P(P(u)v)P(w)<_{\zdp} P(u)P(v)P(w),\\
		&\overline{\phi_1(u,vP(w))}=P(u)P(vP(w))<_{\zdp} P(u)P(v)P(w),\\
		&\overline{\phi_1(u,P(v)w)}=P(u)P(P(v)w)<_{\zdp} P(u)P(v)P(w),\\
		&\overline{P(\phi_1(u,v)w)}=P(P(u)P(v)w)<_{\zdp} P(u)P(v)P(w),\\
		&\overline{P(u\phi_1(v,w))}=P(uP(v) P(w))<_{\zdp} P(u)P(v)P(w).
	\end{eqnarray*}
	So we obtain
	$$(f,g)_{t}^{\mu,\nu} \equiv 0 \;\mod(S_{\Phi_{\rm I}}(Z), t).$$
	
Next consider the underlined ambiguity  $t=D(P(u)P(v))=\overline{\phi_2(u,v)}=D(\star)|_{\phi_1(u,v)}$. We have the composition:
\begin{align*}
&	\phi_2(P(u),P(v))-D(\phi_1(u,v)) \\
=&  -DP(u)P(v)-P(u)DP(v)+DP(uP(v))+DP(p(u)v)\\
=& -\phi_3(u)P(v)-PD(u)P(v)-P(u)\phi_3(v) -P(u)PD(v)  +\phi_3(uP(v))+ PD(uP(v))\\
&+\phi_3(P(u)v)+ PD(P(u)v)\\
=& -\phi_3(u)P(v)-\phi_1(D(u),v) -P(D(u)P(v))-P(PD(u)v)
 -P(u)\phi_3(v)\\
 &-\phi_1(u,D(v))-P(uPD(v))-P(P(u)D(v)) + PD(uP(v))+ PD(P(u)v)\\
=& -\phi_3(u)P(v)-\phi_1(D(u),v)-P(u)\phi_3(v)-\phi_1(u,D(v))  +P(\phi_2(u,P(v))+P(\phi_2(P(u),v)))\\
 \equiv&\, 0 \mod(S_{\Phi_{\rm II}}(Z),t).
	\end{align*}
	Thanks to
\begin{eqnarray*}
		&\overline{\phi_3(u)P(v)}=DP(u)P(v) <_{\zdp} D(P(u)P(v)),\\
		& \overline{\phi_1(D(u),v)}=PD(u)P(v)  <_{\zdp} D(P(u)P(v)),\\
		&\overline{P(u)\phi_3(v)}=P(u)DP(v) <_{\zdp} D(P(u)P(v)),\\
		& \overline{\phi_1(u,D(v))}=P(u)PD(v)  <_{\zdp} D(P(u)P(v)),\\
		&\overline{P(\phi_2(u,P(v)))}=PD(uP(v))  <_{\zdp} D(P(u)P(v)),\\
		&\overline{P(\phi_2(P(u),v))}=PD(P(u)v)  <_{\zdp} D(P(u)P(v)),
	\end{eqnarray*}
we obtain
	$$(f,g)_{t}^{\mu,\nu} \equiv 0 \;\mod(S_{\Phi_{\rm I}}(Z), t).$$
	This completes the proof of type I case. 
	
The cases of type II and type III are similarly verified. 
\end{proof}

\section{Free noncommutative \pdrbas}
\mlabel{s:explicit}
In this section, we apply the Gr\"obner-Shirshov bases for \pdrbas to completely determine the free noncommutative \pdrbas.
Of the three types of \pdrbas given in Definition~\mref{de:pdrbtype} we will treat types I and II together. Then the approach with be modified to treat type III.

\subsection{A linear basis of free \pdrbas}
Let $\rmT\in \{\rmI,\rmII,\rmIII\}$ and $\Omega=\{D,P\}$. Let 
\begin{equation}
\eta\!: \bfk\frakSO(Z) \to \bfk\frakSO(Z)/\Id(S_{\Phi_{T}}(Z)),
\mlabel{eq:eta}
\end{equation} 
be the canonical quotient homomorphism of operated algebras. In the quotient operated algebra $\bfk\frakSO(Z)/\Id(S_{\Phi_{\rmT}}(Z))$, 
define the quotient operations 
$$\overline{D}:= D \mod\Id(S_{\Phi_{T}}(Z)), \quad \overline{P}:= P \mod\Id(S_{\Phi_{T}}(Z)),$$ 
and let $\bar{\ast}$ be the quotient modulo $\Id(S_{\Phi_{T}}(Z))$ of the concatenation on $\bfk\frakSO(Z)$. By Theorem~\mref{thm:quotient}, we obtain
\begin{theorem}
Let $T\in \{\rmI, \rmII, \rmIII\}$. With notations as above, $(\bfk\frakSO(Z)/\Id(S_{\Phi_{T}}(Z)), \bar{\ast}, \overline{D}, \overline{P})$  is the free  \pdrba of type $T$.  
\mlabel{thm:free-pdrba}
\end{theorem}

The Gr\"obner-Shirshov bases for $\Id(S_{\Phi_{T}}(Z))$ obtained in Theorem~\mref{thm:phigsb} gives a linear basis for this free \pdrbas as follows.

\begin{theorem}
Fix $\rmT\in \{\rmI, \rmII, \rmIII\}$. Let
\begin{equation}\Irr(S_{\Phi_{\rmT}}(Z)):=\frakSO(Z)\backslash\{q_1|_{P(u)P(v)},\;q_2|_{D(uv)},\;q_3|_{DP(u)}\,|\,u,v\in \frakSO(Z),\,q_i\in\frakSO^\star(Z), \,1\leq i\leq 3\}.
\mlabel{eq:irrs}
\end{equation}
Then $\eta\big(\Irr(S_{\Phi_{T}}(Z))\big)$ is a linear basis of $\bfk\frakSO(Z)/\Id(S_{\Phi_{T}}(Z))$.
\mlabel{thm:freepara}
\end{theorem}

\begin{proof}
Let $\rmT\in \{\rmI, \rmII, \rmIII\}$. By Theorems\,\mref{thm:phigsb}, $S_{\Phi_\rmT(Z)}$ is a Gr\"obner-Shirshov basis in $\bfk\frakSO(Z)$. Furthermore, the leading terms of $\phi_{\rmT,1}(u, v)$,  $\phi_{\rmT,2}(u, v)$ and $\phi_{\rmT,3}(u)$ in $S_{\Phi_{\rmT}}(Z)$ from Eqs.\,\meqref{eq:phi1} -- \meqref{eq:phi3} under $\leq_{\zdp}$ are
	$ P(u)P(v),  D(uv)$ and $D(P(u))$,  respectively.
	Therefore, by Theorems~\mref{thm:CDL},  $\eta\big(\Irr(S_{\Phi_{\rmT}}(Z))\big)$ is a linear basis of $\bfk\frakSO(Z)/\Id(S_{\Phi_{\rmT}}(Z))$ .
\end{proof}

We next make the linear basis $\Irr(S_{\Phi_{\rmT}}(Z))$ explicit by identifying it with a subset of $S_{\Phi_{\rmT}}(Z)$ obtained in the previous study of differential Rota-Baxter algebras~\mcite{GK3,GS}.
For this purpose, we first recall some notations from~\mcite{Gub,GK3,GS}. 
To be consistent with the previous literature and to avoid later conflict of notations, we will use $X$ instead of $Z$ as the generating set of the free operated algebra $\frakSO(X)$. 

For any subsets $Y$ and $Z$ of $\frakSO(X)$, we define the \textit{alternating product } of $Y$ and $Z$ by 
$$ \Lambda(Y,Z):=\left(   \bigcup_{r\geq1}(YP(Z))^{r}\right)\bigcup \left(\bigcup_{r\geq0}(YP(Z))^{r}Y\right)\bigcup \left(   \bigcup_{r\geq1}(P(Z)Y)^{r}\right)\bigcup \left(\bigcup_{r\geq0}(P(Z)Y)^{r}P(Z)\right).$$

We then construct a sequence of sets by the following recursion.
Firstly, define $$\mathfrak{X}_0:= S(\Delta(X)).$$
for $\Delta(X):= \{ D^n(x)\, \big|\, x\in X, n\geq 0\}.$
In general, for $n\geq1$, we define
\begin{equation}
\mathfrak{X}_n:=\Lambda(\mathfrak{X}_0,\mathfrak{X}_{n-1}).
\mlabel{eq:xn}
\end{equation}
So $\mathfrak{X}_n$ can be explicitly expressed as a disjoint union as follows.
{\small
\begin{align*}\mathfrak{X}_n &=\left(   \bigcup_{r\geq1}(\mathfrak{X}_0P(\mathfrak{X}_{n-1}))^{r}\right)\bigcup \left(\bigcup_{r\geq0}(\mathfrak{X}_0P(\mathfrak{X}_{n-1}))^{r}\mathfrak{X}_0\right)
\bigcup \left(   \bigcup_{r\geq1}(P(\mathfrak{X}_{n-1})\mathfrak{X}_0)^{r}\right)\bigcup \left(\bigcup_{r\geq0}(P(\mathfrak{X}_{n-1})\mathfrak{X}_0)^{r}P(\mathfrak{X}_{n-1})\right).
\end{align*}
}
By Eq.~\meqref{eq:xn}, we obtain $ \mathfrak{X}_0\subseteq \mathfrak{X}_1$.
For a given $k\geq 1$, assume that  $\mathfrak{X}_{n-1}\subseteq \mathfrak{X}_n$ for all $n\leq k$, and so this leads to
$$ \frakX_k=\Lambda (\mathfrak{X}_0, \mathfrak{X}_{k-1}) \subseteq \Lambda (\mathfrak{X}_0, \mathfrak{X}_k)=\frakX_{k+1}.$$
We then get a direct system. Then by taking the direct limit, we obtain a set

\begin{equation} \mathfrak{X}_{\infty}=\bigcup_{n\geq0}\mathfrak{X}_n= \lim_{\longrightarrow}\mathfrak{X}_n.
\mlabel{eq:xinf}
\end{equation}
Elements in $\mathfrak{X}_{\infty}$ are called {\bf differential Rota-Baxter bracketed words}, or simply {\bf DRBWs} \mcite{GS}.

For every DRBW $u\in \fx$,  $u$ has a unique {\bf standard decomposition}
\begin{equation}\mlabel{eq:deco}
u=u_1\cdots u_k,
\end{equation}
where each $u_i$ lies alternatively in $\frakX_0$ or in $P(\fx)$ with $1\leq i\leq k$. This is analogous to \cite[Lemma~4.2.5(c)]{Gub}. We call $k$  the {\bf breadth} of $u$, denoted by $\br(u)$.
Also, let
$$
\dep_P(u):=\min\{k\,|\, u\in\frakX_k\},
$$
and call it the {\bf $P$-depth} of $u$. Denote by $\bfk\fx$  the free module on the set $\fx$.

\begin{prop}
\mlabel{p:xinfirr}
	\begin{equation}\mlabel{eq:frakX}
		\frakX_\infty=\Irr(S_{\Phi_{T}}(X)).
	\end{equation}
\end{prop}

\begin{proof}
Let $S_{\phi_{T}}:=S_{\Phi_{T}}(X)$ and let $\Id_{\Phi_T}:=\Id(S_{\Phi_{T}})=\Id(S_{\Phi_{T}}(X))$. Let $\eta\!: \bfk\frakSO(X) \to \bfk\frakSO(X)/\Id_{\Phi_T}$ be the canonical homomorphism for $T\in \{I, II\}$.  Recall from  Eq.~\meqref{eq:irrs} that
\begin{equation*}\Irr(S_{\Phi_{T}}(X)):=\frakSO(X)\backslash\{q_1|_{P(u)P(v)},\;q_2|_{D(uv)},\;q_3|_{DP(u)}\,|\,u,v\in \frakSO(X),\,q_i\in\frakSO^\star(X), \,1\leq i\leq 3\}.
\end{equation*}
For $n\geq 0$, we set
$$I_n:=\{u\in\Irr(S_{\Phi_T})\,|\, \dep_P(u)\leq n\}\subseteq \frakSO(X).$$
Thus, $I_n\subseteq I_{n+1}$ for all $n\geq0$, and 
\begin{equation}\mlabel{eq:in}
\Irr(S_{\Phi_{T}})=\bigcup_{n\geq 0}I_n. 
\end{equation}
By Eq.~\meqref{eq:xn}, we have $\dep_P(u)\leq n$ for all $u\in\frakX_n$, and $u$ doesn't contain $P(x)P(y), D(xy)$ and $DP(x)$ as its subword for $x,y\in \frakSO(X)$. So $\frakX_n\subseteq I_n$.

On the other hand, we prove $I_n\subseteq \frakX_n$ by induction on $n\geq0$. For any $u\in I_0$, we get $\dep_P(u)=0$, and so $u\in \frakS_{\{D\}}(X)$. Then  by Eq.~\meqref{eq:decom0}, $u=v_1v_2\cdots v_k\in I_0$,  where $v_i\in X\cup \lc \frakSO(X)\rc_{\Omega}$ for 
$1\leq i\leq k$ and $\Omega=\{D\}$.   
Since $u$ cannot contain $D(xy)$ as its subword, it follows that $v_i \in X\cup D^n(X)$ for $n\geq 1$. Thus, $u\in S(\Delta(X))=\frakX_0$, and so $I_0\subseteq\frakX_0$. Assume that $I_n\subseteq \frakX_n$ for an $n\geq 0$. Consider $u\in I_{n+1}$. By Eq.~\meqref{eq:decomp}, we write 
\begin{equation*}
u=v_0\lc v_1^\ast\rc_{P} v_1 \cdots \lc v_k^\ast \rc_{P} v_k, \quad v_0, v_1\cdots, v_k\in
M(X)\cup\lc\frakS_{\Omega}(X)\rc_{D}, \quad v_1^\ast,\cdots, v_k^\ast \in \frakSO(X)
\end{equation*}
Assume that $v_{t_0}=\lc \overline{v_{t_0}}\rc_D\in \lc\frakS_{\Omega}(X)\rc_{D}$ for some $t_0\in\{0,\cdots,k\}$. Since $u$ doesn't contain $P(x)P(y), DP(x)$ and $D(xy)$ as its subword for $x,y\in\frakSO(X)$,  we have $\overline{v_{t_0}}\in \Delta(X)$. This shows that $v_i\in M(X)\cup \lc\Delta(X)\rc_D$. For all $v_j^\ast\in \frakSO(X)$ with $j=1,\cdots, k$, we have $\dep_P(v_{j}^\ast)\leq n$, and so $v_j^\ast\in \frakX_n$ by the induction hypothesis. Thus, $\lc v_j^\ast\rc_P\in \lc\frakX_n\rc_P$. This gives $u\in\frakX_{n+1}$, and hence $I_{n+1}\subseteq \frakX_{n+1}$, completing the induction.

Then by Eq.~\meqref{eq:in}, we complete the proof.
\end{proof}

\nc{\NC}{\mathrm{NC}}

Let $\iota:\bfk \frakX_\infty \to \bfk\frakSO(X)$ be the natural inclusion. 
Set $\theta= \eta \circ \iota$.
Let
$$\sha_q^{\NC}(X):=\bfk\fx.$$
 Then we obtain the following commutative diagram.
\begin{equation}
\begin{split}
\xymatrix{
\sha_q^{\NC}(X)
\ar@{^{(}->}[r]^{  \iota}\ar^{\theta}[rd]&(\bfk\frakSO(X),D,P)
\ar@{->}^{\eta}[d]\\
&(\bfk\frakSO(X)/\Id(S_{\Phi_{T}}(X)), \bar{\ast}, \overline{D}, \overline{P})
}\end{split}
\end{equation}

By Proposition~\mref{p:xinfirr}, we obtain
\begin{lemma}
\mlabel{lem:theat-dpx}
With  notations as above,  we have the linear isomorphism
\begin{equation}
	\theta: \bfk \frakX_{\infty} \xrightarrow{\iota}  \bfk\frakSO(X) \xrightarrow{\eta} \bfk\frakSO(X)/\Id(S_{\Phi_{T}}(X)).
	\mlabel{eq:theta}
\end{equation}
\end{lemma}

\subsection{Explicit constructions of \pdrbas of types I and II}

We can use the linear bijection $\theta$ in Eq.\,\meqref{eq:theta} to transport the multiplication, differential operator and Rota-Baxter operator on the free \pdrba $(\bfk\frakSO(X)/\Id(S_{\Phi_{T}}(X)), \bar{\ast}, \overline{D}, \overline{P})$ to $\sha^{\NC}(X):=\bfk \frakX_\infty$, arriving at the isomorphism 
\begin{equation}
\Big(\bfk\frakSO(X)/\Id(S_{\Phi_{T}}(X)), \bar{\ast}, \overline{D}, \overline{P}\Big) \cong 	
\Big(\sha^{\NC}(X),\theta^{-1}\bar{\ast}(\theta\ot \theta), \theta^{-1}D\theta, \theta^{-1}P\theta\Big)
\mlabel{eq:freeiso}
\end{equation} 
of free \pdrbas on $X$. Here $\bar{\ast}$ is the multiplication on $\bfk\frakSO(X)/\Id(S_{\Phi_{\rmT}}(X))$ induced by the concatenation on $\bfk\frakSO(X)$ in Theorem~\mref{thm:free-pdrba} and 
$$\theta^{-1}\bar{\ast}(\theta\ot \theta)(u\ot v)=
\theta^{-1}(\theta(u) \bar{\ast} \theta(v)).$$

Our goal is to give explicit descriptions of the operations $\theta^{-1}(\bar{\ast}(\theta\ot \theta)), \theta^{-1}D\theta, \theta^{-1}P\theta$, 
 thereby obtain an explicit construction of free \pdrbas. 

We first describe the transported Rota-Baxter operator $\theta^{-1}\overline{P}\theta$. 

Note that if $u$ is in $\fx$, then $P(u)$ is still in $\fx$. Thus we have a linear map 
\begin{equation}
\pfx: \bfk\fx\to \bfk \fx, \quad \pfx(u):=P(u).
\mlabel{eq:rbox}
\end{equation}

\begin{lemma}\mlabel{lem:qpx} We have 
	\begin{equation}\mlabel{eq:theta-pfx}
		\pfx= \theta^{-1}\overline{P} \theta.
	\end{equation}
\end{lemma}

\begin{proof}
	Since $\overline{P}\eta= \eta P$ and $\iota(u)=u$ for all $u\in\frakX_\infty$, we have 
	$$(\theta^{-1}\overline{P} \theta)(u)=(\theta^{-1}\overline{P} \eta \iota)(u)=\theta^{-1}\overline{P} \eta (u)=\theta^{-1}\eta (P  (u))=\theta^{-1}\eta \iota (P  (u))=\theta^{-1} \theta (P  (u))
	=\pfx(u).  \qedhere $$
\end{proof}

We next define a binary operation $\shpr$ on $\sha_q^{\rm NC}(X)$ as the transported product of $\bar{\ast}$. We use the following recursion in analog to that for the product of two Rota-Baxter words in the free Rota-Baxter algebra on a set. See~\cite[Theorem~4.2.16]{Gub} for more details.

Consider basis elements  $u,v\in\fx$. Using the standard decomposition from Eq.~\eqref{eq:deco}, write 
$$
u=u_1\cdots u_m,\quad v=v_1\cdots v_n,
$$
with $\br(u)=m\geq 1$ and $\br(v)=n\geq 1$. We define $u\shpr v$ by induction on the sum  $\ell:=\dep_P(u)+\dep_P(v)\geq 0$.
For the base case $\ell=0$, we have $\dep_P(u)=\dep_P(v)=0$, and so $u,v$ are in $\frakX_{0}=S(\Delta(X))$. In this case, we define $u\shpr  v=uv$, the concatenation product.

For a given $k\geq 0$, assume that $u \shpr  v$ has been defined for all $u,v\in \frakX_{\infty}$ with $0\le \ell\le k$. Consider $u, v\in \frakX_\infty$ with $\ell=\dep_P(u)+\dep_P(v)=k+1$. We distinguish two cases, depending on whether or not $\br(u)=\br(v)=1$.

\noindent
{\bf Case 1. }
Let $\br(u)=\br(v)=1$, that is, $u$ and $v$ are in $\frakX_0$ or in $P(\frakX_{\infty})$. Then
\begin{equation}\mlabel{eq:de}
	u \shpr  v = \left \{\begin{array}{ll}
		P(\bar{u} \shpr  P(\bar{v}))+P(P(\bar{u})\shpr  \bar{v})+\lambda P(\bar{u} \shpr  \bar{v}),
		& \text{if} \ u=P(\bar{u})\  \text{and} \ v=P(\bar{v}),\\
		u v, & \text{otherwise}.
	\end{array} \right .
\end{equation}
Here the product in the first line is well-defined by the induction hypothesis, since
$$\dep_P(\bar{u})+\dep_P(P(\bar{v}))=\dep_P(P(\bar{u}))+\dep_P(\bar{v})=k\quad\text{and}\quad \dep_P(\bar{u})+\dep_P(\bar{v})=k-1.$$
The product in the second line is the concatenation.

\noindent
{\bf Case 2. }
Let $\br(u)>1$ or $\br(v)>1$. Then define
\begin{equation}\mlabel{eq:dd}
	u \shpr  v=u_1\cdots u_{m-1}(u_m\shpr v_1)v_2\cdots v_{n},
\end{equation}
where $w_m \shpr  v_1$ has been  defined by Eq.~(\mref{eq:de}) and the remaining products are the concatenation. Extending $\shpr $ by bilinearity, we obtain  a binary operation $\shpr $ on $\bfk\fx$. 

\begin{lemma}
\mlabel{lem:thetaprod}	
For the product $\shpr$ defined above, we have
\begin{equation}
\mlabel{eq:thetaprod}
	u \shpr v= \theta^{-1}\big(\theta(u) \bar{\ast} \theta(v)\big), \quad u, v\in \fx.
\end{equation}
\end{lemma}

\begin{proof}
Use induction on the sum $\ell:=\dep_P(u)+\dep_P(v)\geq 0$.	 If $\ell=0$, then $u,v\in S(\Delta(X))\subseteq\frakX_\infty$, and so
$u\shpr v=uv$. Since $\iota (a)=a$ for all $a\in\frakX_\infty$, we have 
$$\theta(u)\bar{\ast} \theta(v)=\eta(u)\bar{\ast} \eta(v)=\eta(uv)=\theta(uv).$$
Hence $uv=\theta^{-1}(\theta(u)\bar{\ast}\theta(v))$. Thus, Eq.~\meqref{eq:thetaprod} holds in this case.
Now suppose Eq.~\meqref{eq:thetaprod} holds for all $u,v\in\frakX_\infty$ with $\ell\leq k$. Consider $u,v\in\frakX_\infty$ with $\ell=k+1$. There are two cases to examine. 
\noindent

\smallskip
{\bf Case 1.}
Let $\br(u)=\br(v)=1$. Then either $u=P(\bar{u}), v=P(\bar{v})$ for some $\bar{u},\, \bar{v}\in\frakX_\infty$, or one of $u, v$ belongs to $S(\Delta(X))$. For the former case, we have
\begin{align*}\theta(u)\bar{\ast}\theta(v)&=\eta(P(\bar{u}))\bar{\ast}\eta(P(\bar{v}))\\
&=\overline{P}(\eta(\bar{u}))\bar{\ast}\overline{P}(\eta(\bar{v}))\quad(\text{by $\eta  P=\overline{P} \eta$})\\
&=\overline{P}\Big(\eta(\bar{u})\bar{\ast}\overline{P}(\eta(\bar{v}))+\overline{P}(\eta(\bar{u}))\bar{\ast}\eta(\bar{v})+\lambda\eta(\bar{u})\bar{\ast}\eta(\bar{v})\Big)\\
&=\overline{P}\Big(\theta(\bar{u})\bar{\ast}\theta(P(\bar{v}))+\theta(P(\bar{u}))\bar{\ast}\theta(\bar{v})+\lambda\theta(\bar{u})\bar{\ast}\theta(\bar{v})\Big)\\
&=\overline{P}\Big(\theta(\bar{u}\shpr P(\bar{v}))+\theta(P(\bar{u})\shpr \bar{v})+\lambda\theta(\bar{u}\shpr \bar{v})\Big)\quad(\text{by the induction hypothesis})\\
&=\overline{P}\Big(\eta(\bar{u}\shpr P(\bar{v}))+\eta(P(\bar{u})\shpr \bar{v})+\lambda\eta(\bar{u}\shpr \bar{v})\Big)\\
&=\eta\Big(P\big(\bar{u}\shpr P(\bar{v})+P(\bar{u})\shpr \bar{v}+\lambda\bar{u}\shpr \bar{v}\big)\Big)\\
&=\eta(u\shpr v)\quad(\text{by Eq.~\meqref{eq:de}})\\
&=\theta(u\shpr v)\quad(\text{by $\iota(u\shpr v)=u\shpr v\in\bfk\frakX_\infty$}). 
\end{align*}
For the latter case, we obtain $uv\in\frakX_\infty$, and hence $\iota(uv)=uv$. Then 
$$\theta(u)\bar{\ast}\theta(v)=\eta(u)\bar{\ast}\eta(v)=\eta(uv)=\theta(uv).$$
By Eq.~\meqref{eq:de} again,  $u\shpr v=uv$, and so Eq.~\meqref{eq:thetaprod} holds.
\noindent

\smallskip
{\bf Case 2.} Suppose $\br(u)\geq 2$ or $\br(v)\geq 2$.  Write $u$ and $v$ in their standard decompositions:
$$u=u_1\cdots u_m,\quad v=v_1\cdots v_n,$$
where each $u_i$, $1\leq i\leq m$, and each $v_j$, $1\leq j\leq n$, lies alternatively in $\frakX_0$ or in $P(\fx)$.
Then
\begin{align*}
\theta(u)\bar{\ast}\theta(v)&=\eta(u)\bar{\ast}\eta(v)\\
&=\eta(u_1\cdots u_{m-1})\bar{\ast}\big(\eta(u_m)\bar{\ast}\eta(v_1)\big)\bar{\ast}\eta(v_2 \cdots v_n)\\
&=\eta(u_1\cdots u_{m-1})\bar{\ast}\theta(u_m\shpr v_1)\bar{\ast}\eta(v_2 \cdots v_n)\quad(\text{by Case 1})\\
&=\eta(u_1\cdots u_{m-1})\bar{\ast}\eta(u_m\shpr v_1)\bar{\ast}\eta(v_2 \cdots v_n)(\text{by $\iota(u\shpr v)=u\shpr v\in\bfk\frakX_\infty$})\\
&=\eta(u_1\cdots u_{m-1}u_m\shpr v_1v_2 \cdots v_n)\\
&=\eta(u\shpr v)\quad(\text{by Eq.~\meqref{eq:dd}})\\
&=\theta(u\shpr v)\quad (\text{by $\iota(u\shpr v)=u\shpr v\in\bfk\frakX_\infty$}),
\end{align*}
proving Eq.~\meqref{eq:thetaprod}.
\end{proof}

Finally, we define a linear operator $\dx$ on $\bfk\fx$ that will be identified with $\theta^{-1}\overline{D}\theta$. It suffices to define $\dx(u)$ for basis elements $u\in \fx$.  We now  proceed by induction on  $\dep_P(u)=n\geq 0$. If $n=0$, then  $\frakX_0=S(\Delta(X))$, and so $u\in S(\Delta(x))$. We can write  $u=u_1\cdots u_m$ with $u_i\in \Delta(X)$ and $m\geq 1$. We then define $\dx(u)$ by induction on $m\geq 1$.
If $m=1$, then $u\in \Delta(X)$, and  define
\begin{equation}\mlabel{eq:291}
\dx(u):=D(u).
\end{equation}
By the induction hypothesis, for  $m\geq 2$, we define
\begin{equation} \mlabel{eq:301}
    \dx(u):=\dx(u_1u_2\dots u_m)=\dx(u_1)u_2\cdots u_m+u_1\dx(u_2\cdots u_m)+\lambda \dx(u_1)\dx(u_2\cdots u_m).
\end{equation}

Suppose  $\dx(u)$ has been defined for all $u\in \frakX$ with $\dep_P(u)\leq n$ for a given $n\geq 0$. Consider $u\in \frakX_\infty$ with $\dep_P(u)=n+1$. By Eq.~\meqref{eq:deco} and $\frakX_{n+1}=\Lambda(\frakX_0,\frakX_n)$, $u=u_1\cdots u_m$ such that $u_i$ is alternatively in $\frakX_0(=S(\Delta(X))$ or in $P(\frakX_n)$. When $m=1$ and $u_i\in \frakX_0$, $\dx(u)$ has been defined by the initial step of $n=0$.  When $m=1$ and $u\in P(\fx)$, we write $u=P(\bar{u})$ for some  $\bar{u}\in\frakX_n$. Then define
\begin{equation}\mlabel{eq:dpx}
\dx(u):=\pfx(\dx(\bar{u})),
\end{equation}
where $\dx(\bar{u})$ is well-defined by the induction hypothesis, since $\dep_P(\bar{u})\leq n$.
If $m\geq 2$, then we also define recursively $\dx(u)$ by Eqs.~\meqref{eq:301} and ~\meqref{eq:dpx}.

\begin{lemma}
\mlabel{lem:dxt}
For the linear operator $\dx$ defined above,  we have
\begin{equation}\mlabel{eq:dxt}
\dx=\theta^{-1}\overline{D}\theta.
\end{equation}
\end{lemma}

\begin{proof}
For $u\in\frakX_\infty$, by Eq.~\meqref{eq:deco}, we write $=u_1u_2\dots u_m$, where $u_i$ is alternatively in $\frakX_0$ or $P(\frakX_\infty)$ and $m\geq 1$. 
Then Eq.~\meqref{eq:dxt} is equivalent to
\begin{equation}\mlabel{eq:dx}
\overline{D}(\theta (u))=\theta(\dx(u)),\quad u\in\frakX_\infty.
\end{equation}
We next prove this by induction on $\dep_P(u)\geq 0$.

\noindent
{\bf Step 1.} For the base case of $\dep_P(u)=0$, we have $u=u_1\cdots u_m\in S(\Delta(X))$ with $u_i\in\Delta(X)$ and $1\leq i\leq m$. We then proceed by a second induction on $m\geq 1$. For the base step $m=1$, by the equation $\eta(D(a))=\overline{D}(\eta(a))$ for all $a\in\frakX_\infty$ and Eq.~\meqref{eq:291}, we have 
$$\overline{D}(\theta(u))=\overline{D}(\eta(u))=\eta(D(u))=\theta(\dx(u)),$$
proving ~\meqref{eq:dx}.
Suppose Eq.~\meqref{eq:dx} holds for  $m\leq k$ (and $\dep_P(u)=0$). Consider $m=k+1$. Then
\begin{align*}\overline{D}(\theta(u))
&=\overline{D}(\eta(u_1 u_2\cdots u_m)\\
&=\overline{D}(\eta(u_1)\bar{\ast} \eta(u_2\cdots u_{m}))\\
&=\overline{D}(\eta(u_1))\bar{\ast}\eta(u_2\cdots u_{m-1})+\eta(u_1)\bar{\ast}\overline{D}(\eta(u_2\cdots u_{m})+\lambda \overline{D}(\eta(u_1))\bar{\ast}\overline{D}(\eta(u_2\cdots u_m))\\
&=\eta(\dx(u_1))\bar{\ast}\eta(u_2\cdots u_{m-1})+\eta(u_1)\bar{\ast}\eta(\dx(u_2\cdots u_{m}))+\lambda \eta(\dx(u_1))\bar{\ast}\eta(\dx(u_2\cdots u_m))\\
&\quad(\text{by the induction hypothesis on }m)\\
&=\eta(\dx(u_1)(u_2\cdots u_{m-1})+u_1\dx(u_2\cdots u_{m})+\lambda \dx(u_1)\dx(u_2\cdots u_m))\\
&=\eta(\dx(u))\quad(\text{by Eq.~\meqref{eq:301}})\\
&=\theta(\dx(u))\quad(\text{by $\dx(u)\in\bfk\frakX_\infty$}).
\end{align*}

\noindent
{\bf Step 2.} 
Assume that Eq.~\meqref{eq:dx} holds for $\dep_P(u)\leq n$. Consider $\dep_P(u)=n+1$. Write $u$ in its standard decomposition as
$$u=u_1\cdots u_m,$$
where each $u_i$, $1\leq i\leq m$,  lies alternatively in $\frakX_0$ or in $P(\frakX_n)$. We again proceed by induction on $m$. If $m=1$ and $u\in S(\Delta(X))$, then Eq.~\meqref{eq:dx} holds by the initial step of $n=0$. If $m=1$ and $u=P(\bar{u})$ for some $\bar{u}\in\frakX_n$ with $\dep_P(\bar{u})\leq n$, then by the equation $\overline{D}\;\overline{P}=\overline{P}\;\overline{D}$,
\begin{align*}
\overline{D}(\theta(u))&=\overline{D}(\overline{P}(\eta(\bar{u})))\\
&=\overline{P}(\overline{D}(\theta(\bar{u})))\\
&=\overline{P}(\theta(\dx(\bar{u}))\quad(\text{by the induction hypothesis on }\dep_P(u))\\
&=\theta(\pfx(\dx(\bar{u})))\quad(\text{by Eq.~\meqref{eq:theta-pfx}})\\
&=\theta(\dx(P(\bar{u})))\quad(\text{by Eq.~\meqref{eq:dpx}})\\
&=\theta(\dx(u)),
\end{align*}
proving Eq.~\meqref{eq:dx}.
Suppose Eq.~\meqref{eq:dx} holds for all $u\in\frakX_\infty$ with $m\leq k$ and $\dep_P(u)=n+1$. Now consider $u\in\frakX_\infty$ with  $m=k+1$. We then deduce that
\begin{align*}
\overline{D}\big(\theta(u_1u_2\dots u_m)\big)
&=\overline{D}\big(\eta(u_1)\bar{\ast}\eta(u_2\cdots u_m)\big)\\
&=\overline{D}\big(\eta(u_1)\big)\bar{\ast}\eta(u_2\cdots u_m)+\eta(u_1)\bar{\ast}\overline{D}\big(\eta(u_2\cdots u_m)\big)+\lambda\overline{D}\big(\eta(u_1)\big)\bar{\ast}\overline{D}\big(\eta(u_2\cdots u_m)\big)\\
&=\theta(\dx(u_1))\bar{\ast}\eta(u_2\cdots u_m)+\eta(u_1)\bar{\ast}\theta(\dx(u_2\cdots u_m))+\lambda\theta(\dx(u_1))\bar{\ast}\theta(\dx(u_2\cdots u_m))\\
&\quad(\text{by $u_1, \,u_2\cdots u_m\in\frakX_\infty$  and the induction hypotheses on } m \text{ and } \dep_P(u))\\
&=\eta(\dx(u_1))\bar{\ast}\eta(u_2\cdots u_m)+\eta(u_1)\bar{\ast}\eta(\dx(u_2\cdots u_m))+\lambda\eta(\dx(u_1))\bar{\ast}\eta(\dx(u_2\cdots u_m))\\
&\quad(\text{by $\dx(u_1), \,\dx(u_2\cdots u_m)\in\frakX_\infty$})\\
&=\eta(\dx(u))\quad(\text{by Eq.~\meqref{eq:301}})\\
&=\theta(\dx(u))\quad(\text{by $\dx(u)\in\bfk\frakX_\infty$}),
\end{align*}
proving Eq.~\meqref{eq:dxt}.
\end{proof}

By Theorem~\mref{thm:free-pdrba} and Lemmas~\mref{lem:qpx} -- \mref{lem:dxt}, we obtain the following main result for \pdrba of types I and II. 

\begin{theorem} Let $X$ be a set.
The quadruple $(\sha_q^{\NC}(X),\shpr,\dx,\pfx)$ along with the inclusion map $j_{X}: X\to \sha_q^{\NC}(X)$ is the free  \pdrba of types I and II on $X$. 
\mlabel{thm:main-type1}
\end{theorem}

\begin{remark} When we take $\lambda=0$ in the case (i): $b=0$, then it becomes the case when $\lambda=0$, $b= 0$, which is just the case (i) discussed in Section~\mref{s:gsbqdrb}.
That is, $P$ and $Q$ satisfies the following conditions:
\begin{enumerate}
\item   $P(x) P(y)=P(x P(y))+P(P(x) y)$,
\item   $D(xy)=D(x)y+xD(y)$,
\item  $DP(x)=PD(x)$.
\end{enumerate}
This becomes the case of \pdrbas of type I.
\end{remark}

\subsection{Explicit construction of free \pdrbas of type III}  
We finally modify the explicit construction of free \pdrbas of type I and II to the case of type III, namely when $\lambda=0, b\neq 0$. As in the case of type I and II, we still start with the isomorphism 
\begin{equation}
\Big(\bfk\frakSO(X)/\Id(S_{\Phi_{T}}(X)), \bar{\ast}, \overline{D}, \overline{P}\Big) \cong 	
\Big(\sha^{\NC}(X),\theta^{-1}\bar{\ast}(\theta\ot \theta), \theta^{-1}\overline{D}\theta, \theta^{-1}\overline{P}\theta\Big)
\mlabel{eq:freeiso3}
\end{equation} 
of free \pdrbas on $X$, as in Eq.~\meqref{eq:freeiso}. 
We then give identify the operations on $\sha^{\NC}(X)$ with explicitly defined operations $\pfx, \dx$ and $\diamond$. 

The constructions of $P_\frakX$ and $\diamond$ are exactly the same. So we will focus on the construction of $\dx$ and identify it with $\theta^{-1}\overline{D}\theta$. 

To define $\dx$, it suffices to specify $\dx(u)$ for each basis element $u \in \mathfrak{X}$, for which we proceed by induction on the depth $\dep_P(u)=n \geq 0$.

For the base depth $n=0$, we have $u \in \mathfrak{X}_0 = S(\Delta(X))$.  
Write $u = u_1 \cdots u_m$ with $u_i \in \Delta(X)$ and $m \geq 1$, and then define $\dx(u)$ by induction on $m$:
If $m=1$, then  $u \in \Delta(X)$, and set $\dx(u) = D(u)$.  
If $m \geq 2$, define recursively  
\begin{equation} \mlabel{eq:302}
    \dx(u)=\dx(u_1u_2\dots u_m)=\dx(u_1)u_2\cdots u_m+u_1\dx(u_2\cdots u_m).
\end{equation}

For the inductive step on depth, we assume $\dx(u)$ has been defined for all $u$ with $\dep_P(u) \leq n$.  
Take $u \in \mathfrak{X}_\infty$ with $\dep_P(u) = n+1$.  
By Eq.~\eqref{eq:deco} and the description $\mathfrak{X}_{n+1} = \Lambda(\mathfrak{X}_0,\mathfrak{X}_n)$, we may write  
$$
u = u_1 \cdots u_m,
$$
where the factors alternate between $\mathfrak{X}_0$ and $P(\mathfrak{X}_n)$.
We then proceed by induction on $m\geq 1$. If $m=1$ and $u \in \mathfrak{X}_0$, the definition is already covered by the base depth $n=0$.
If $m=1$ and $u = P(\bar u)$ with $\bar u \in \mathfrak{X}_n$, define  
\begin{equation}\mlabel{eq:dpx2}
\dx(u)=\pfx(\dx(\bar{u}))+bP(\bar{u}),
\end{equation}
  which is well-defined, because $\dep_P(\bar u) \leq n$.
If $m \geq 2$, define $\dx(u)$ recursively by combining Eqs.~\meqref{eq:302} and ~\meqref{eq:dpx2}.

With this modification and analogs to Lemmas~\mref{lem:qpx} -- \mref{lem:dxt}, we obtain the following main result for type III \pdrbas.

\begin{theorem}\mlabel{thm:ncfqdba}Let $X$ be a set.
The \pdrba $(\sha_{q,0}^{\NC}(X),\shpr,\qd,\qp)$ together with the set embedding $j_X:X\rar \sha_{q,0}^{\NC}(X)$ is the free  \pdrba of type III on $X$.
\end{theorem}

\noindent {\bf Acknowledgments}: This work is supported by the NNSFC (12461002,12326324), the Jiangxi Provincial Natural Science Foundation (20224BAB201003).
S. Zheng thanks the Chern Institute of Mathematics at Nankai University for hospitality.

\noindent
{\bf Declaration of interests. } The authors have no conflicts of interest to disclose.

\noindent
{\bf Data availability. } Data sharing is not applicable as no data were created or analyzed.

\end{document}